\documentclass[reqno]{amsart}
\usepackage[all]{xy}
\usepackage{color}
\usepackage{mathrsfs}
\usepackage{hyperref}
\usepackage{enumerate}
\usepackage{amsmath}
\usepackage{amsthm}
\usepackage{amssymb}
\usepackage{stmaryrd}
\usepackage{amscd}
\usepackage{graphicx}
\usepackage{epsfig}
\usepackage{combelow}
\usepackage[english]{babel}
\usepackage[stable]{footmisc}
\usepackage{stackrel}

\renewcommand{\d}{\mathrm d}               
\newcommand{\Lie}{\pounds}    
\newcommand{\ab}{\mathrm{ab}} 
\newcommand{\ext}{\mathrm{ext}}  
\newcommand{\eps}{\varepsilon}  
\newcommand{\genus}{\approx} 

   \renewcommand{\a}{\alpha}

\newcommand{\R}{\mathbb{R}}

\newcommand{\Ss}{\mathbb{S}}
\newcommand{\T}{\mathbb{T}}
\newcommand{\Z}{\mathbb{Z}}

   \newcommand{\cG}{\mathcal{G}}

    \newcommand{\cH}{\mathcal{H}}

  \newcommand{\cN}{\mathcal{N}}

\newcommand{\s}{\mathbf{s}}             
\renewcommand{\t}{\mathbf{t}}           
\newcommand{\A}{A}                      
\renewcommand{\gg}{\mathfrak{g}}        
\newcommand{\hh}{\mathfrak{h}}          

\newcommand{\tto}{\rightrightarrows}    

\DeclareMathOperator{\Ker}{Ker}           
\DeclareMathOperator{\coKer}{coKer}           
\DeclareMathOperator{\im}{Im}           
\renewcommand{\Im}{\im}
\DeclareMathOperator{\Ad}{Ad}           
\DeclareMathOperator{\Per}{Per}           
\DeclareMathOperator{\SPer}{SPer}           
\DeclareMathOperator{\pr}{pr}      

\allowdisplaybreaks

\newtheorem{theorem}{Theorem}[section]
\newtheorem{lemma}[theorem]{Lemma}
\newtheorem{proposition}[theorem]{Proposition}
\newtheorem{corollary}[theorem]{Corollary}

\theoremstyle{definition}
\newtheorem{definition}[theorem]{Definition}
\newtheorem{example}[theorem]{Example}

\newtheorem{remark}[theorem]{Remark}

\begin{document}
\title[Genus Integration]{Genus Integration, Abelianization and Extended Monodromy}

\author{Ivan Contreras}

\address{Department of Mathematics \& Statistics, Campus Box 2239, Amherst College,
Amherst, MA 01002}
\email{icontreraspalacios@amherst.edu}

\author{Rui Loja Fernandes}
\address{Department of Mathematics, University of Illinois at Urbana-Champaign, 1409 W. Green Street, Urbana, IL 61801 USA}
\email{ruiloja@illinois.edu}

\thanks{RLF was partially supported by NSF grants DMS-1405671, DMS-1710884, and a Simons Fellowship in Mathematics.}

\begin{abstract}
Given a Lie algebroid we discuss the existence of a smooth abelian integration of its abelianization. We show that the obstructions are related to the extended monodromy groups introduced recently in \cite{CFMb}. We also show that this groupoid can be obtained by a path-space construction, similar to the Weinstein groupoid of \cite{CF1}, but where the underlying homotopies are now supported in surfaces with arbitrary genus. {As an application, we show that the prequantization condition for a (possibly non-simply connected) manifold is equivalent to the smoothness of an abelian integration.} Our results can be interpreted as a generalization of the classical Hurewicz theorem.
\end{abstract}

\maketitle

\setcounter{tocdepth}{1}
\tableofcontents


\section{Introduction}

A complete solution to the problem of integrating a Lie algebroid $A\to M$ to a Lie groupoid $\cG\tto M$ was given in \cite{CF1}. The crucial idea underlying this solution is the path-space construction, which has its roots in the Poisson sigma-model \cite{CaF,CF2,Ikeda,SS94} and Sullivan's rational homotopy theory \cite{Severa}: given a Lie algebroid $A$ one considers the space of $A$-paths $P(A)$, a certain Banach manifold. On $P(A)$ there is an equivalence relation $\sim$ given by the so-called $A$-homotopies, and the quotient space:
\[ \cG(A):=P(A)/\sim \]
is a topological groupoid, called the {\bf Weinstein groupoid}. Moreover, there exists a Lie groupoid with Lie algebroid $A\to M$ if and only if the Weinstein groupoid $\cG(A)\tto M$ has a smooth structure for which the projection $P(A)\to\cG(A)$ is a submersion. In this case, $\cG(A)$ is the source 1-connected integration of $A$. In \cite{CF1}, the obstructions to integrability are also obtained (see below).

The notion of $A$-homotopy plays a crucial role in the construction of the Weinstein groupoid. Recall that an {\bf $A$-path} can be thought of as a Lie algebroid morphism $a:TI\to A$, where $I=[0,1]$. Similarly, an {\bf $A$-homotopy} between $A$-paths $a_0$ and $a_1$ is a Lie algebroid morphism $h:T(I\times I)\to A$, such that $h|_{TI\times\{i\}}=a_i$. We will extend the notion of $A$-homotopy to allow for {\bf $A$-homologies}, which are Lie algebroid morphisms $T\Sigma\to A$, where instead of the square $I\times I$ we allow for surfaces $\Sigma$ of arbitrary genus. This gives a coarser equivalence relation $\genus$ on $P(A)$ and we define the {\bf genus integration} of $A$ to be the topological groupoid:
\[ \cG_g(A)=P(A)/\genus. \]
There is an obvious surjective groupoid morphism $p:\cG(A)\to\cG_g(A)$ and the following two natural questions arise:
\begin{itemize}
\item What is the geometric meaning of the genus integration?
\item Can the genus integration be smooth, i.e., a Lie groupoid? If yes, what are the obstructions to genus-integrability?
\end{itemize}
This paper is devoted to give complete answers to these questions. It is well-known that one can think of $\cG(A)$ as the first homotopy group(oid) of the generalized space represented by $A$. In brief, our results show that one should think of $\cG_g(A)$ as the first homology group(oid) of the generalized space represented by $A$.

An {\bf abelian Lie algebroid} is an algebroid whose isotropy Lie algebras are all abelian. Given a Lie algebroid $A\to M$ its {\bf abelianization} is given by a morphism $p:A\to A^\ab$ to an abelian Lie algebroid which is universal among morphisms into abelian algebroids: for any \emph{abelian} Lie algebroid $B\to N$ and any Lie algebroid morphism $\phi:A\to B$, there is a unique Lie algebroid morphism $\overline{\phi}:A^\ab\to B$ such that:
\[
\xymatrix{
A\ar[d]_p\ar[r]^{\phi} & B \\
A^\ab \ar@{-->}[ru]_{\overline{\phi}} 
}
\]
The abelianization may fail to exist. However, when $A\to M$ is transitive, the abelianization $A^\ab\to M$ always exists and is given by factoring out the commutators of the isotropy Lie algebras.

Similarly, an {\bf abelian groupoid} is a groupoid whose isotropy groups are all abelian.  Given a (Lie) groupoid $\cG\tto M$ one defines its {\bf abelianization} as a morphism $p:\cG\to\cG^\ab$ to an abelian (Lie) groupoid satisfying a similar universal property: for any abelian (Lie) groupoid $\cH$ and any morphism $\Phi:\cG\to\cH$ there is a unique morphism $\overline{\Phi}:\cG^\ab\to \cH$ such that:
\[
\xymatrix{
\cG\ar[d]_p\ar[r]^{\Phi} & \cH \\
\cG^\ab \ar@{-->}[ru]_{\overline{\Phi}} 
}
\]
Note that the abelianization depends on which category one works. In the category of sets, any groupoid has an abelianization, which is obtained by factoring out the commutators of the isotropy groups. But in the smooth category, a Lie groupoid may fail to have an abelianization, and if it exists it may be different from the set-theoretical abelianization. 

Our first result states that $p:\cG(A)\to\cG_g(A)$ can be thought of as a generalization of the Hurewicz homomorphism (to recover the classical case take $A=TM$):

\begin{theorem}
\label{thm:main:1:int}
Let $A\to M$ be a Lie algebroid. The genus integration $\cG_g(A)$ is the (set-theoretical) abelianization of the Weinstein groupoid $\cG(A)$.
\end{theorem}

Note that there is no assumption about integrability in the statement of this theorem. We will see that when $A$ is integrable and $\cG(A)$ is a Lie groupoid, the genus integration $\cG_g(A)$ may fail to be smooth. 

In \cite{CF1} it is shown that that the smoothness of $\cG(A)$, and hence the integrability of $A$, is controlled by the so-called {\bf monodromy groups} $\cN_x(A)$. These appear as the image of a {\bf monodromy homomorphism}:
 \[ \partial_x:\pi_2(L,x)\to \cG(\gg_x), \]
 where $L\subset M$ is the leaf of $A$ through $x$ and  $\cG(\gg_x)$ is the 1-connected Lie group integrating the isotropy Lie algebra $\gg_x$. One of the main results in \cite{CF1} states that $\cG(A)$ is smooth if and only the groups $\cN_x(A)$ are uniformly discrete.
 
 Following \cite{CFMb}, we will introduce an {\bf extended monodromy homomorphism}:
 \[ \partial_x^\ext:H_2(\tilde{L}^h)\to \cG(\gg_x^\ab), \]
 where $\tilde{L}^h$ is a certain holonomy cover of the leaf $L$. The image of this homomorphism is the {\bf extended monodromy group} $\cN^\ext_x(A)$. These different monodromies are related via a commutative diagram:
\[
\xymatrix@R=15pt{
\pi_2(L,x)\ar[d]\ar[r]^{\partial_x} & \cG(\gg_x)\ar[d]\\
H_2(\tilde{L}^h)\ar[r]_{\partial_x^\ext}& \cG(\gg_x^\ab)
}
\]
where the first column is induced from the Hurewicz map. Our second main theorem states that:

\begin{theorem}
\label{thm:main:2:int}
Let $A\to L$ be a transitive Lie algebroid with trivial holonomy. The following statements are equivalent:
\begin{enumerate}[(a)]
\item the genus integration $\cG_g(A)$ is smooth;
\item the extended monodromy groups are discrete;
\item the abelianization $A^\ab$ has an abelian integration.
\end{enumerate}
Moreover, if any of these hold, then $\cG_g(A)$ has Lie algebroid isomorphic to $A^\ab$.
\end{theorem}

 
 As an immediate corollary we obtain:
 
 \begin{corollary}
If $A\to M$ is a transitive Lie algebroid with trivial holonomy whose monodromy and extended monodromy are both discrete, then the genus integration $\cG_g(A)$ is the abelianization of the Weinstein groupoid $\cG(A)$ in the smooth category.
 \end{corollary}
 
When the holonomy is not trivial, the situation is a bit more tricky, and the extended monodromy obstructs the smoothness of the genus integration of a certain algebroid cover of $A$. We refer to Section \ref{section:smooth:genus:integration} for details. 

Using our results, we can deduce that the extended monodromy obstructs the existence of proper integrations, generalizing the results of \cite{CFMb}:

\begin{theorem}
If a Lie algebroid $A\to M$ admits a proper integration $\cG\tto M$ then all its extended monodromy groups are discrete.
 \end{theorem}

{
Another consequence of our results concerns a geometric interpretation of the {\bf prequantization condition}: given a closed 2-form $\omega$ on a manifold $M$, this condition can be stated in terms of the group of periods of $\omega$ as:
\[ \Per(\omega):=\{\int_\gamma \omega:[\gamma]\in H_2(M,\Z)\}\subset \R\text{ is a discrete subgroup}. \]
On the other hand, a closed 2-form $\omega\in\Omega^2(M)$ is a Lie algebroid 2-cocycle on the tangent bundle $TM$ and hence defines a central extension $A_\omega=TM\oplus\R$. The group of periods of $\omega$ coincides with the extended monodromy groups of $A_\omega$. Using the genus integration we obtain the following extension of a result of \cite{Cr}:

\begin{theorem}
Let $\omega\in\Omega^2(M)$ be a closed 2-form. The following conditions are equivalent:
\begin{enumerate}[(i)]
\item $\omega$ is prequantizable: $\Per(\omega)=a\Z$, for some $a\in\R$;
\item the genus integration of $A_\omega$ is smooth;
\item there is a principal $\Ss^1_a$-bundle $\pi:P\to M$ with connection $\theta$ such that:
\[ \pi^*\omega=\d\theta. \]
\end{enumerate}
\end{theorem}

When $M$ is simply connected, the groups of periods and spherical periods of $\omega$ coincide and we recover the result of \cite{Cr}.

}

This paper is organized as follows. In Section 2, we study the abelianization of Lie algebroids and Lie groupoids. In Section 3, we introduce the genus integration of a Lie algebroid and we prove Theorem \ref{thm:main:1:int}. Section 4 contains our definition of extended monodromy and proves some of its properties, including the relationship with ordinary monodromy. Finally, in Section 5 we prove Theorem \ref{thm:main:2:int} and its corollaries. Section 6 contains a list of some open problems that naturally arises from our work and which we hope to address in the future. Appendix A contains an auxliar result about integration of algebroid morphisms.
\medskip

{\bf Acknowledgments.}  This paper evolved from independent discussions, on the one hand, with Owen Gwilliam and Alan Weinstein, on the genus integration in the special case of a Lie algebra and, on the other hand, with Marius Crainic, Ioan M\u{a}rcu\cb{t} and David Martinez-Torres, on the notion of extended monodromy for Poisson manifolds. We thank all of them for generously sharing their ideas with us, as well as the two anonymous referees for their detailed comments.

\section{Abelianization of algebroids and groupoids}


\subsection{Lie algebroids}

Let $A\to M$ be a Lie algebroid with anchor $\rho:A\to TM$ and Lie bracket $[~,~]:\Gamma(A)\times\Gamma(A)\to\Gamma(A)$. Recall that we say that $A$ is {\bf abelian} if the isotropy Lie algebras $\gg_x=\Ker\rho_x$ are abelian for all $x\in M$. 

\begin{definition}
\label{def:abelianization:algebrd}
Given a Lie algebroid $A\to M$ its {\bf abelianization} is an abelian Lie algebroid $A^\ab\to M$ together with a {surjective morphism $p:A\to A^\ab$ covering the identity}  such that: for any \emph{abelian} Lie algebroid $B\to N$ and any Lie algebroid morphism $\phi:A\to B$, there is a unique Lie algebroid morphism $\overline{\phi}:A^\ab\to B$ that makes the following diagram commute: 
\[
\xymatrix{
A\ar[d]_p\ar[r]^{\phi} & B \\
A^\ab \ar@{-->}[ru]_{\overline{\phi}} 
}
\]
\end{definition}

A standard argument shows that, if a Lie algebroid $A$ has an abelianization $A^\ab$, then it is unique up to isomorphism. 

\begin{lemma}
If $A^\ab\to M$ is the abelianization of $A\to M$, then for any $x\in M$ one has $\Im\rho_x=\Im\rho^\ab_x$ and
\[ [\gg_x,\gg_x]\subset \Ker p_x. \]
\end{lemma}

\begin{proof}
{The identity $\Im\rho_x=\Im\rho^\ab_x$ follows from the fact that $p:A\to A^\ab$ is a surjective Lie algebroid morphism. On the other hand,} If $a_1,a_2\in\gg_x$, we can choose sections $\a_i\in\Gamma(A)$ such that $\a_i(x)=a_i$. Then $p\circ\a_i\in\Gamma(A^\ab)$ are sections such that $p(\a_i(x))\in\Ker\rho^\ab_x$. Moreover, since $A^\ab$ is abelian and $p:A\to A^\ab$ is a morphism, we find:
\[ p([a_1,a_2])=p([\a_1,\a_2](x))=[p\circ\a_1,p\circ\a_2](x)=[p(\a_1(x)),p(\a_2(x))]=0, \]
so the statement follows.
\end{proof}

\begin{example}
\label{ex:2dim:algbrd}
For a Lie algebra $\gg$, viewed as a Lie algebroid over a singleton, the abelianization is the Lie algebra $\gg^\ab=\gg/[\gg,\gg]$. However, already a bundle of Lie algebras $A\to M$ may have an abelianization whose fiber \emph{is not} $\gg_x/[\gg_x,\gg_x]$. For example, for the bundle of Lie algebras $A=\R\times\R^2\to\R$, with coordinate $x$ on the base and Lie bracket:
\[ [e_1,e_2]=x e_2. \]
$A^\ab\to \R$ is the trivial rank one bundle of Lie algebras and $A^\ab_0 \ne  \gg^\ab_0=\gg_0$.
\end{example}

\begin{example}
\label{ex:3dim:algbrd}
Consider the infinitesimal action of the Lie algebra $\mathfrak{so}(3)$ on $\R^3$ by rotations. The corresponding action Lie algebroid $A=\mathfrak{so}(3)\times\R^3\to \R^3$ does not have a abelianization. {In fact, if we assume that $p:A\to A^\ab$ existed, then at the origin we would have:
\[ \mathfrak{so}(3)=[\mathfrak{so}(3),\mathfrak{so}(3)]\subset \Ker p_0. \]
Then $A^\ab$ would have rank $0$, contradicting $\Im\rho=\Im\rho^\ab$.}
\end{example}

For a transitive Lie algebroid $A\to M$ we will denote by $\gg_M=\Ker\rho$ its isotropy bundle, so we have the short exact sequence of Lie algebroids:
\begin{equation}
\label{transitive:seq}
\xymatrix{0\ar[r] & \gg_M\ar[r] & A\ar[r]^\rho & TM \ar[r]& 0.}
\end{equation}
The bundle of Lie algebras $\gg_M$ has a subalgebroid $[\gg_M,\gg_M]$ consisting of the bundle of Lie algebras with fiber the commutator:
\[ [\gg_M,\gg_M]|_x=[\gg_x,\gg_x]. \]
For these, it is easy to see that that the abelianization always exists. The following proposition is due to Mackenzie (see \cite{Mackenzie87,Mackenzie05}):

\begin{proposition}
\label{prop:abel:algbrd}
Let $A\to M$ be a transitive Lie algebroid. Then its abelianization is $A^\ab=A/[\gg_M,\gg_M]$. In particular,  $\Ker\rho^\ab_x\simeq \gg^\ab_x$, for all $x\in M$.
\end{proposition}

\begin{proof}
The bundle of Lie algebras $[\gg_M,\gg_M]$ forms an ideal system over the identity in $A\to M$, according to the terminology of \cite[Section 4.4]{Mackenzie05}. Hence, the quotient $A/[\gg_M,\gg_M]$ has a Lie algebroid structure such that the projection $p:A\to A/[\gg_M,\gg_M]$ is a surjective morphism. It is easy to see that it satisfies the universal property, so that $A^\ab=A/[\gg_M,\gg_M]$.

An alternative approach is to observe that choosing a splitting $\sigma:TM\to A$ we obtain a Lie algebroid isomorphism:
\[ A\simeq TM\oplus\gg_M,\]
where the anchor becomes the projection in $TM$ and the Lie bracket is given by:
\[ [(X,\xi),(Y,\eta)]_A=([X,Y],[\xi,\eta]_{\gg_M}+\nabla^\sigma_X\eta-\nabla^\sigma_Y\xi+\Omega^{\sigma}(X,Y)). \]
Here $\Omega^{\sigma}\in\Omega^2(M,\gg_M)$ denotes the curvature of the splitting defined by:
\[ \Omega^{\sigma}(X,Y)=[\sigma(X),\sigma(Y)]-\sigma([X,Y]), \]
and $\nabla^\sigma$ denotes the connection on $\gg_M$ given by $\nabla^\sigma_X\a:=[\sigma(X),\a]$.

Denoting by $\gg^\ab_M$ the bundle of abelian Lie algebras $\gg_M/[\gg_M,\gg_M]$, we can build a new algebroid $A^\ab\to M$ by setting:
\[ A^\ab:= TM\oplus\gg^\ab_M,\]
with anchor the projection in $TM$ and Lie bracket given by:
\[ [(X,\xi),(Y,\eta)]_A^\ab:=([X,Y],[\xi,\eta]_{\gg^\ab_M}+\nabla^{\sigma^\ab}_X\eta-\nabla^{\sigma^\ab}_Y\xi+\Omega^{\sigma^\ab}(X,Y)), \]
where $\Omega^{\sigma^\ab}=p\circ \Omega^{\sigma}\in\Omega^2(M,\gg^\ab_M)$ and $\nabla^{\sigma^\ab}$ is a connection on $\gg^\ab_M$ induced from $\nabla^\sigma$ under the projection
$p:\gg_M\to\gg^\ab_M$. It is easy to check that $A^\ab$ satisfies the universal property.
\end{proof}

\begin{remark}
{
We have assume in Definition \ref{def:abelianization:algebrd} that the morphism $p:A\to A^\ab$ is surjective. We conjecture that this is actually a consequence of the universal property but we were not able to prove it. Note that we have used surjectivity both in the proof of the lemma above and in the examples to determine the abelianization.
}
\end{remark}


\subsection{Lie groupoids}

Let $\cG\tto M$ be a Lie groupoid with source $\s$ and target $\t$.  Recall that we say that $\cG$ is {\bf abelian} if the isotropy Lie groups $\cG_x=\s^{-1}(x)\cap \t^{-1}(x)$ are abelian for all $x\in M$. 

In the next definition the terms ``groupoid" and ``groupoid morphism" can be interpreted in either the category of sets, topological spaces, or smooth manifolds.

\begin{definition}
Given a groupoid $\cG\tto M$ its {\bf abelianization} is an abelian groupoid $\cG^\ab\to M$ together with a {surjective groupoid morphism $p:\cG\to \cG^\ab$ covering the identity} such that: for any \emph{abelian} groupoid $\cH\tto N$ and any groupoid morphism $\Phi:\cG\to \cH$, there is a unique groupoid morphism $\overline{\Phi}:\cG^\ab\to \cH$ that makes the following diagram commute:
\[
\xymatrix{
\cG\ar[d]_p\ar[r]^{\Phi} & \cH \\
\cG^\ab \ar@{-->}[ru]_{\overline{\Phi}} 
}
\]
\end{definition}

If a groupoid $\cG$ has an abelianization $\cG^\ab$, then it is unique up to isomorphism. 

\begin{example}
If $\cG\tto M$ is a groupoid in sets, then its abelianization always exists: 
\[ \cG^\ab=\cG/(G_M,G_M),\] 
where $G_M=\bigcup_{x\in M} \cG_x$ is the bundle of isotropies and $(G_M,G_M)$ is the bundle formed by their commutators. 
\end{example}

\begin{example}
If $\cG\tto M$ is a topological groupoid, then its abelianization is: 
\[ \cG^\ab=\cG/\overline{(G_M,G_M)},\] 
where $\overline{(G_M,G_M)}$ denotes the closure of the commutator bundle.
\end{example}

\begin{example}
\label{ex:Lie:grp}
If $G$ is a Lie group, then its abelianization always exists: 
\[ G^\ab=G/\overline{(G,G)},\] 
where $\overline{(G,G)}$ denotes the closure of the commutator subgroup. For a connected Lie group $G$ the commutator $(G,G)$ coincides with the connected Lie subgroup integrating the commutator Lie subalgebra $[\gg,\gg]\subset\gg$ (see \cite[$\S$XII]{Chevalley}). If $G$ is 1-connected, then $(G,G)$ is the kernel of the Lie group morphism integrating the projection $\gg\to \gg/[\gg,\gg]$, hence is a closed subgroup. 

However, in general, $(G,G)\subsetneq \overline{(G,G)}$ as shown by the following example (see \cite[Chp.~3, $\S 9$, Exer.~9]{Bourbaki}): 
\[ G=(\widetilde{\mathrm{SL}}(2,\R)\times \Ss^1)/\Z, \]
where $\Z$ is embedded diagonally as the normal subgroup of the center of the universal cover $\widetilde{\mathrm{SL}}(2,\R)$ whose quotient is ${\mathrm{SL}}(2,\R)$ and as a non-discrete subgroup of $\Ss^1$ (for example, one can take $\{e^{\pi n\sqrt{2} i}:n\in \Z\}\subset \Ss^1$).
\end{example}

\begin{example}
\label{ex:fund:grpd}
Let $\Pi_1(M)\tto M$ be the fundamental groupoid of a connected manifold $M$. If $\widetilde{M}$ denotes the universal covering space, this groupoid can be identified with the quotient of the pair groupoid $\widetilde{M}\times\widetilde{M}$ by the action of the fundamental group $\pi_1(M)$:
\[ \Pi_1(M)= (\widetilde{M}\times\widetilde{M})/\pi_1(M). \]
The isotropy groups are isomorphic to $\pi_1(M)$, so this groupoid is non-abelian, in general. To describe its abelianization, let 
$\overline{M}=\widetilde{M}/(\pi_1(M),\pi_1(M))$, 
be the cover of $M$ with covering group $H_1(M,\Z)=\pi_1(M)^\ab$. Then:
\[ \Pi_1(M)^\ab= (\overline{M}\times\overline{M})/H_1(M,\Z). \]
\end{example}

\begin{example}
Consider the bundle of Lie groups $\cG=(0,+\infty)\times \R^2\tto \R$ integrating the Lie algebroid from Example \ref{ex:2dim:algbrd}:
\begin{itemize}
\item source and target maps: $\s(a,b,x)=x=\t(a,b,x)$;
\item multiplication: $(a_1,b_1,x)\cdot (a_2,b_2,x)=(a_1a_2,a_1^xb_2+b_1,x)$. 
\end{itemize}
Its abelianization is the bundle of Lie groups $\cG=(0,+\infty)\times\R\tto\R$ where:
\begin{itemize}
\item source and target maps: $\s(a,x)=x=\t(a,x)$;
\item multiplication: $(a_1,x)\cdot (a_2,x)=(a_1a_2,x)$. 
\end{itemize}
At $x=0$, we have that $\cG^\ab_0\not=(\cG_0)^\ab.$
\end{example}

\begin{example}
\label{ex:SO3:action}
Let $\cG=\mathrm{SO}(3,\R)\times\R^3\tto \R^3$ be the action groupoid associated with the usual action of $\mathrm{SO}(3,\R)$ on $\R^3$ by rotations. Then a reasoning similar to Example \ref{ex:3dim:algbrd} shows that $\cG^\ab$ does not exist (in the category of Lie groupoids). 
\end{example}

These examples show that, if $\cG$ is a Lie groupoid with Lie algebroid $A$, then:
\begin{itemize}
\item the abelianization $\cG^\ab$ may fail to exist (Example \ref{ex:SO3:action});
\item if $\cG^\ab$ exists, then its Lie algebroid, in general, is not isomorphic to the abelianization $A^\ab$ (Example \ref{ex:Lie:grp}).
\item even if $\cG$ is source 1-connected, $\cG^\ab$ may be different from the source 1-connected integration $\cG(A^\ab)$ (Example \ref{ex:fund:grpd}).
\end{itemize}

Let us focus now on the case of a transitive Lie groupoid $\cG\tto M$. Then the isotropy groups form a bundle of Lie groups $G_M\to M$, which is a closed normal Lie subgroupoid of $\cG$, and we have a short exact sequence of Lie groupoids:
\begin{equation}
\label{transitive:seq:grpd}
\xymatrix{1\ar[r] & G_M\ar[r] & \cG \ar[r]^--{(\s,\t)} & M\times M \ar[r]& 0.}
\end{equation}
The bundle of Lie groups $G_M$ has the normal Lie subgroupoid $(G_M,G_M)$, a bundle of Lie groups with fibers the commutators $(\cG_x,\cG_x)$. The following result shows that for a transitive Lie groupoid the abelianization always exists.

\begin{proposition}
\label{prop:abel:grpd}
Let $\cG\tto M$ be transitive Lie groupoid with algebroid $A\to M$. The closure $\overline{(G_M,G_M)}$ is a normal Lie subgroupoid of $\cG$ and one has:
\[ \cG^\ab=\cG/\overline{(G_M,G_M)}. \]
Moreover, the following are equivalent:
\begin{enumerate}[(a)]
\item $(G_M,G_M)\subset\cG$ is a closed subgroupoid;
\item for all $x\in M$, $(\cG_x,\cG_x)$ is a closed subgroup of $\cG_x$;
\item for some $x\in M$, $(\cG_x,\cG_x)$ is a closed subgroup of $\cG_x$;
\end{enumerate}
If any of these hold and $\cG_x$ is connected, then $\cG^\ab$ has Lie algebroid $A^\ab$.
\end{proposition}

\begin{proof}
Given a transitive Lie groupoid $\cG\tto M$, fix $x\in M$. Then $\t:\s^{-1}(x)\to M$ is a principal $\cG_x$-bundle, and we have an isomorphism between $\cG\tto M$ and the gauge groupoid:
\[ (\s^{-1}(x)\times \s^{-1}(x))/\cG_x\tto M. \]
Under this isomorphism, the isotropy bundle becomes isomorphic to the associated bundle:
\[ G_M\simeq (\s^{-1}(x)\times \cG_x)/\cG_x, \]
where $\cG_x$ acts on itself by conjugation. Then one sees immediately that the commutator bundle and its closure are isomorphic to the associated bundles:
\[ (G_M,G_M)\simeq (\s^{-1}(x)\times (\cG_x,\cG_x))/\cG_x,\quad \overline{(G_M,G_M)}\simeq (\s^{-1}(x)\times \overline{(\cG_x,\cG_x)})/\cG_x. \]
This proves that $\overline{(G_M,G_M)}$ is a closed, normal, Lie subgroupoid of $\cG$ and so the quotient $\cG/\overline{(G_M,G_M)}$ is a Lie groupoid (see \cite[Section 2.4]{Mackenzie05}). One checks easily that it satisfies the universal property. Also, the equivalences between (a)-(c) follows from the description of the commutator and its closure as associated bundles, and Proposition \ref{prop:abel:algbrd}. When $\cG_x$ is connected, the commutator $(\cG_x,\cG_x)$ coincides with the connected Lie subgroup integrating the commutator Lie subalgebra $[\gg,\gg]\subset\gg$ (see \cite[$\S$XII]{Chevalley}), and it follows that the Lie algebroid of $\cG^\ab$ is isomorphic to $A^\ab$.
\end{proof}

In general, even when $\cG$ is source 1-connected and $(G_M,G_M)\subset\cG$ is closed, the Lie algebroid of $\cG^\ab$ may not be isomorphic to $A^\ab$. 

\begin{example}
\label{ex:gauge:groupoid}
Consider the gauge groupoid associated with the  $\mathrm{O}(2)$-principal bundle $\Ss^3\to \mathbb{RP}^2$:
\[ \cG=(\Ss^3\times \Ss^3)/\mathrm{O}(2)\tto \mathbb{RP}^2. \] 
$\cG$ is a source 1-connected groupoid with isotropy groups isomorphic to $\mathrm{O}(2)$. The associated Lie algebroid $A\to \mathbb{RP}^2$ is abelian, so we have $A=A^\ab$ and $\cG=\cG(A)$.

On the other hand, the commutator group $(\mathrm{O}(2),\mathrm{O}(2))=\mathrm{SO}(2)$ is closed in $\mathrm{O}(2)$, so that $(G_M,G_M)\subset\cG$ is closed. We conclude also that the abelianization $\cG^\ab=\cG/(G_M,G_M)$ has discrete isotropy isomorphic to $\Z_2=\mathrm{O}(2)/\mathrm{SO}(2)$, so its Lie algebroid is isomorphic to $T(\mathbb{RP}^2)$. In fact, it is easy to see that $\cG^\ab=\Pi_1(\mathbb{RP}^2)$.
\end{example}


\section{Genus integration}

Let $A\to M$ be a Lie algebroid. In this section we construct its genus integration which is a quotient:
\[ \cG_g(A):=P(A)/\genus, \]
where $\genus$ denotes homology of $A$-paths, which we will now define.

Let us first recall the notions of $A$-paths and $A$-homotopies from \cite{CF1}. By an {\bf $A$-path} we will mean a smooth curve $a:I\to A$ satisfying
\[ \rho(a(t))=\frac{\d }{\d t} \gamma(t), \]
where $\gamma:I\to M$ is the base path of $a$. This condition is equivalent to require that
\[
\xymatrix@R=15pt{
TI\ar[d]\ar[r]^{a\d t} & A\ar[d] \\
I\ar[r]_{\gamma} & M
}
\]
is a Lie algebroid morphism. 

Let $a_0,a_1:I\to A$ be $A$-paths with the same end points: $\gamma_0(i)=\gamma_1(i)=x_i$, $i=0,1$. An {\bf $A$-homotopy} between $a_0$ and $a_1$ is a Lie algebroid morphism 
\[
\xymatrix@C=80pt{
T(I\times I)\ar[d]\ar[r]^{a(t,\eps)\d t+b(t,\eps)\d \eps} & A\ar[d] \\
I\times I\ar[r]_{\gamma(t,\eps)} & M
}
\]
satisfying the boundary conditions:
\begin{equation}
\label{eq:bd:conditions}
\gamma(i,\eps)=x_i, \quad a(t,i)=a_i(t),\quad b(i,\eps)=0,\quad (i=0,1)
\end{equation}

\begin{figure}[h]
    \centering
    \includegraphics[height=5cm]{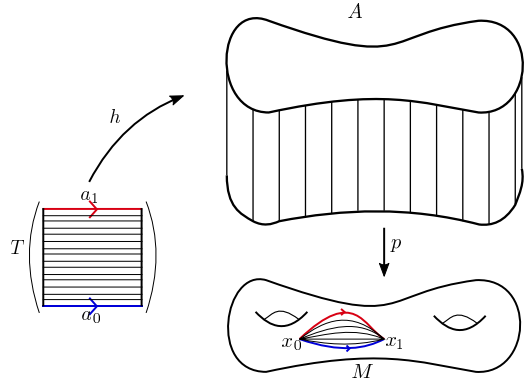}
    \caption{Homotopic $A$-paths}
    \label{AHom}
\end{figure}

Using an auxiliary connection, the condition that $a\d t+b\d \eps$ is a Lie algebroid morphism can be written as $a$ and $b$ satisfying the pde (see \cite{CF1,Mackenzie05} for details):
\begin{equation}\label{diffeq}
\partial_{t}b- \partial_{\epsilon} a= T_{\nabla}(a, b).
\end{equation}

 
{
\begin{remark}
\label{rem:reparemeterization:A:paths}
Homotopy defines an equivalence relation among $A$-paths. Reparameterization of $A$-paths gives rise to equivalent $A$-paths (see \cite{CF1}), so every $A$-path is equivalent to one which has all the derivatives vanishing at the end-points. From now on we will assume that all $A$-paths vanish to all orders at the end-points. In particular, this applies to the $A$-paths in an $A$-homotopy: $t\mapsto a(t,\eps)$ and $\eps\mapsto b(t,\eps)$.
\end{remark}
}

It follows that there is well-defined operation of concatenation of $A$-paths $a_1\cdot a_0$ inducing a partially defined product on homotopy equivalence classes:
\[  [a_1]\cdot [a_2]:=[a_1\cdot a_2]. \]
If we let $\s,\t:P(A)\to M$ be defined by $\s(a)=\gamma(0)$ and $\t(a)=\gamma(1)$, and we set
$\mathbf{u}:M\to P(A)$, $x\mapsto 0_x(t)$, and $\mathbf{i}:P(A)\to P(A)$, $a(t) \mapsto -a(1-t)$, the quotient:
\[ \cG(A):=P(A)/\sim, \]
has a groupoid structure over $M$. Moreover, this is a topological groupoid for the quotient topology induced by the $C^1$-topology on $P(A)$ (\cite[Thm 2.1]{CF1}).
\medskip 

For the definition of $A$-homology we will replace the square $I\times I$ by a \emph{square with genus}, so let us clarify what we mean by this. We start with the square $I\times I$ and we remove some open disk from its interior $D\subset \textrm{int}(I\times I)$. Then we take some compact surface $\Sigma'$ with genus $n$ and with one single boundary component $\partial \Sigma'=\partial D$ and we form the connected sum:
\[ \Sigma=\Sigma'\sqcup_{\partial D} (I\times I)\backslash D. \]
We call $\Sigma$ a \textbf{square with genus}. Its boundary $\partial \Sigma$ has a neighborhood $U$ diffeomorphic to a neighborhood of the boundary of $I\times I$, and henceforth we will use this identification with no further notice, calling $U$ a ``collar neighborhood'' of $\partial \Sigma$. 

\begin{definition}
\label{def:homology}
Let $a_0,a_1:I\to A$ be $A$-paths with the same end points. An {\bf $A$-homology} between $a_0$ and $a_1$ is a Lie algebroid map $h:T\Sigma\to A$, where 
$\Sigma$ is a square with genus, such that in a collar neighborhood $U$ of $\partial \Sigma$:
\[ h|_U=a(t,\eps)\d t+b(t,\eps)\d \eps \] 
satisfy the boundary conditions \eqref{eq:bd:conditions}.
\end{definition}

\begin{figure}[h]
    \centering
    \includegraphics[height=5cm]{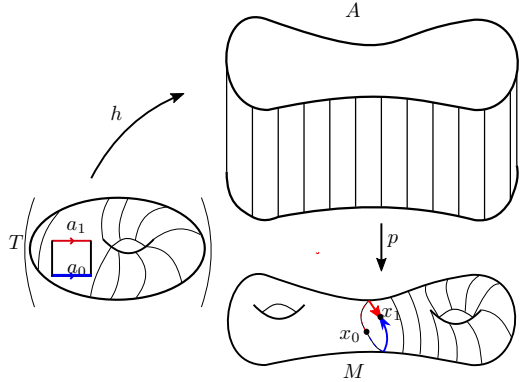}
    \caption{Homologous $A$-paths}
    \label{Hom}
\end{figure}


\begin{remark}
{ 
We have set up the definition of $A$-homology so that possible difficulties with smoothness have already been solved by the usual theory of $A$-homotopies (cf.~\cite{CF1}), since the boundary conditions of a square with genus supporting an $A$-homology are the same as the boundary conditions of an ordinary square supporting an $A$-homotopy. In particular, note that:
\begin{enumerate}[(i)]
\item if we have an $A$-homology $h:T\Sigma\to A$ between $a_1$ and $a_2$, where $\Sigma$ is a square with genus $g$, and an $A$-homology $h':T\Sigma'\to A$ between $a_2$ and $a_3$, where $\Sigma'$ is a square with genus $g'$, then we can glue vertically the two squares  along the side corresponding to $a_2$ to a square with genus $g+g'$ supporting an $A$-homology $T(\Sigma\#\Sigma')\to A$ between $a_1$ and $a_3$;

\item if we have an $A$-homology $h:T\Sigma\to A$ between $a_1$ and $a_2$, where $\Sigma$ is a square with genus $g$, and an $A$-homology $h':T\Sigma'\to A$ between $a_1'$ and $a_2'$, where $\Sigma'$ is a square with genus $g'$, and the compositions $a_1\circ a_1'$ and $a_2\circ a_2'$ are defined, then we can glue horizontally the two squares with genus $\Sigma\#\Sigma'$ to get an $A$-homology between $a_1\circ a_1'$ and $a_2\circ a_2'$.
\end{enumerate}
Also, since all our $A$-paths are assumed to vanish at the end-points { (see Remark \ref{rem:reparemeterization:A:paths}),} in all definitions above the square $I\times I$ can be assumed to have smooth boundary, i.e., it can be assumed to be a closed disk with boundary a circle with 4 marked points.}

\begin{figure}[h]
    \centering
    \includegraphics[height=4.5cm]{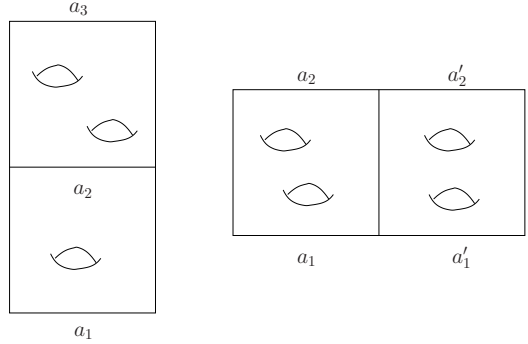}
    \caption{Composition of $A$-homologies}
    \label{squares}
\end{figure}
\end{remark}

We will write $a_0\genus a_1$ if there exists an $A$-homology between the $A$-paths $a_0$ and $a_1$ and we called them {\bf homologous $A$-paths}. It should be clear from the definition that:
\begin{itemize}
\item If $a_0\sim a_1$ then $a_0\genus a_1$;
\item If $a_0\genus a'_0$ and $a_1\genus a'_1$, then $a_0\cdot a_1\genus a'_0\cdot a'_1$;
\item If $a_0,a_1$ belong to the same leaf $L$, then $a_0\genus a_1$ in $A$ iff $a_0\genus a_1$ in $A|_L$.
\end{itemize}
Therefore the space of equivalence classes of homologous $A$-paths (where the genus is not fixed) also has a groupoid structure:

\begin{definition}
The {\bf genus integration} of a Lie algebroid $A\to M$ is the quotient groupoid:
\[ \cG_g(A):=P(A)/\genus. \]
\end{definition}

Notice that there is an obvious surjective groupoid morphism 
\[ p:\cG(A)\to\cG_g(A),\quad [a]\mapsto [a]_g. \] 
One can think of $\cG(A)$ as the first homotopy group(oid) of the generalized space represented by $A$. We will now see that, similarly, one should think of $\cG_g(A)$ as the first homology group(oid) of the generalized space represented by $A$.
\medskip

Recall that the homology of a space is an abelian group. A very similar idea leads to the following:

\begin{proposition}
The genus integration $\cG_g(A)\tto M$ is an abelian groupoid.
\end{proposition}

\begin{proof}
Let $a_1,a_2:I\to A$ be two $A$-loops based at $x_0\in M$. All we need to show is that $(a_1\cdot a_2)\genus (a_2\cdot a_1)$ or, equivalently, that the commutator
\[ a:T(\partial (I\times I))\to A,\quad a:=(a_1,a_2)=a_1\cdot a_2\cdot a_1^{-1} \cdot a_2^{-1}, \]
is $A$-homologous to the trivial path $0_x$. 

For that, let $\Sigma'$ be obtained from the square $I\times I$ by removing some open disk from its interior $D\subset \textrm{int}(I\times I)$:
\[ \Sigma'=(I\times I) \backslash D. \]
Also, let $\phi:\Sigma'\to \partial (I\times I)$ be a retraction of {$\Sigma'$} to the boundary $\partial (I\times I)$. { Here we can assume that the square $I\times I$ is a disk with 4 marked points in the boundary, so that such a smooth retraction exists}. We obtain a Lie algebroid morphism:
\[ h':T\Sigma'\to A,\quad h':=a\circ \d \phi, \]
defining an $A$-homotopy between $a$ and $a':=h'|_{T(\partial D)}$ (see figure).

\begin{figure}[h]
    \centering
        \includegraphics[height=5.5cm]{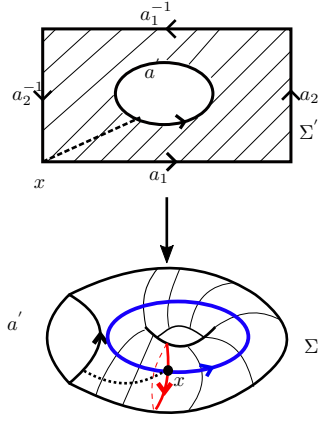}
    \caption{The $A$-paths $a_1\cdot a_2$ and $a_2\cdot a_1$ are homologous}
    \label{GHom}
\end{figure}

If we now glue the opposite sides of the square $\partial(I\times I)$,  we obtain a square with genus $\Sigma$ and {the retraction can be chosen so that} $h'$ induces a Lie algebroid map $h:T\Sigma\to A$ which is an $A$-homology between $a$ and $0_x$. 
\end{proof}

One may wonder why allowing for homotopies with genus results in much more collapsing than ordinary homotopies. Here is one explanation: assume $\cG\tto M$ is an s-connected integration of $A$ and $a_1\sim a_2$ via an $A$-homotopy $h:T(I\times I)\to A$. The pair groupoid is the source 1-connected integration of $T(I\times I)$, and this $A$-homotopy  integrates to a groupoid morphism $H:(I\times I)\times(I\times I)\to \cG(A)$.
Such a groupoid morphism is necessary of the form 
\[ H((t_1,\eps_1),(t_2,\eps_2))=\tilde{h}(t_1,\eps_1)\cdot \tilde{h}(t_2,\eps_2)^{-1}, \]
where $\tilde{h}:I\times I\to \s^{-1}(x)$, $h(0,\eps)=1_x$, is an ordinary homotopy, with fixed end-points, between $g_1,g_2:I\to s^{-1}(x)$, the paths integrating $a_1$ and $a_2$ (\cite[Section 3]{CF3}). Hence, $A$-homotopies lead to identification of paths in source fibers starting at identities that are homotopic and we have $\cG(A)=\tilde{\cG}$, the universal covering of $\cG$.

On the other hand, assume that $a_1\genus a_2$ are homologous $A$-paths via a $A$-homology $h:T\Sigma\to A$. The source 1-connected integration of $T\Sigma$ is not anymore the pair groupoid. If $\Sigma$ has genus $n$, then we may think of it has a $4n$-gon $\Delta_{4n}$ with a disc with boundary $\gamma_1\cdot\gamma_2^{-1}$ removed:

\begin{figure}[h]
    \centering
    \includegraphics[height=5cm]{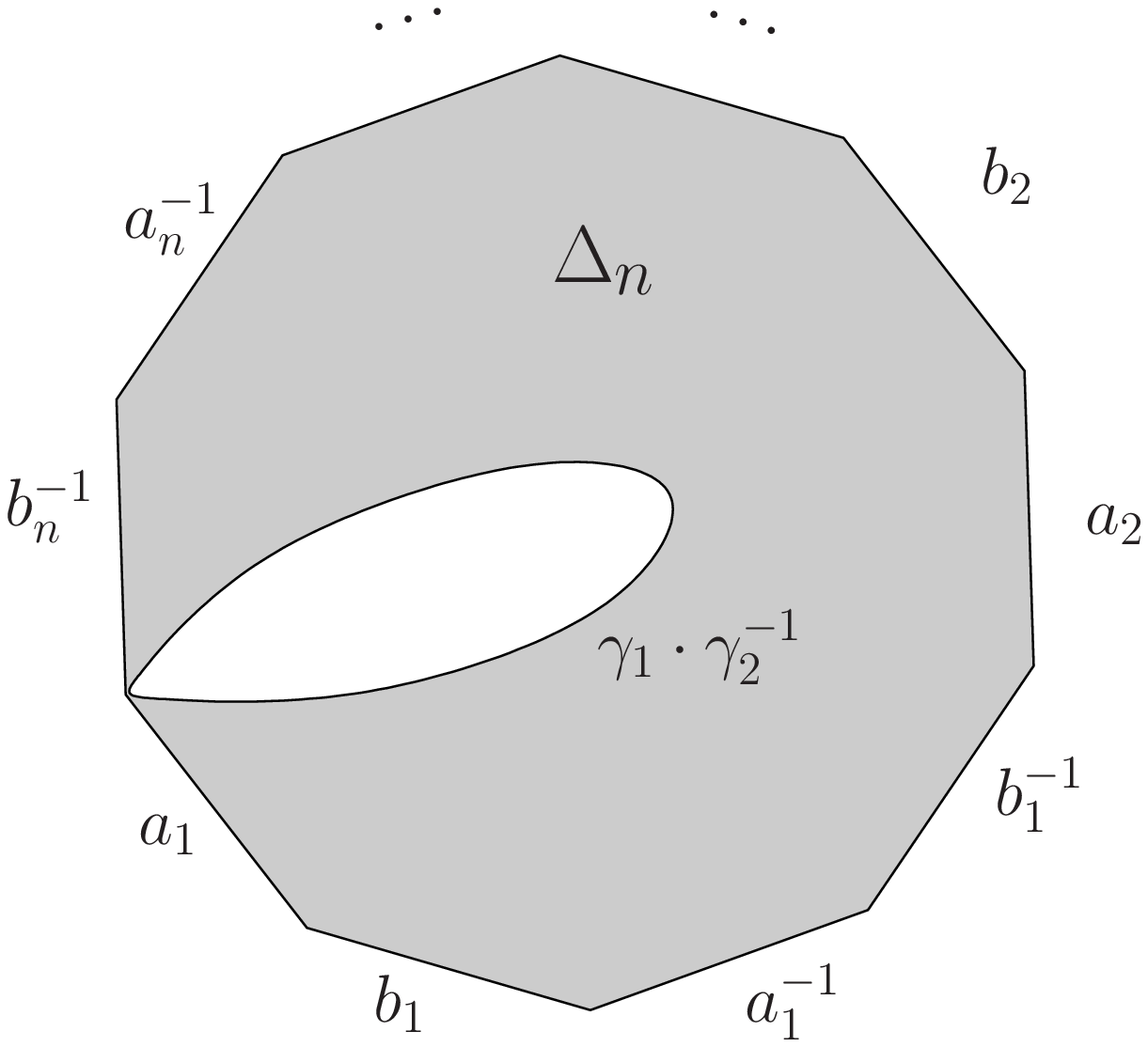}
    \label{ngon}
\end{figure}

Upon integration, we now obtain a homotopy, relative to end-points, $\tilde{h}:\Delta_{4n}\to \s^{-1}(x)$ between the path $g_1\cdot g_2^{-1}:I\to s^{-1}(x)$ integrating $a_1\cdot a_2^{-1}$ and the path $g:I\to (G_x,G_x)$ integrating the $A$-path $h|_{\partial \Delta_{4n}}$. In particular, $g_1\cdot g_2^{-1}$ is not a loop anymore. Hence, $A$-homologies lead to identification of paths in source fibers starting at identities via homotopies where one-end is fixed and the other end is a path in $(G_x,G_x)$. Hence, we find  $\cG_g(A)=\tilde{\cG}/(\tilde{\cG},\tilde{\cG})$.
\medskip

The morphism $p:\cG(A)\to\cG_g(A)$ can be thought of as a generalization of the Hurewicz homomorphism:

\begin{theorem}
\label{thm:main:1}
Let $A\to M$ be a Lie algebroid. The genus integration $\cG_g(A)$ is the set-theoretical abelianization of $\cG(A)$:
\[ \cG_g(A)=\cG(A)/(G_M(A),G_M(A)), \]
where $G_M(A)_x=\bigcup_{x\in M} \cG(A)_x$ is the bundle of isotropies of $\cG(A)$.
\end{theorem}

\begin{proof}
Let $[a]\in\cG(A)_x$ be some $A$-homotopy class. We need to prove that $[a]_g=[0_x]_g$ if and only if $[a]\in (\cG(A)_x,\cG(A)_x)$. Since $\cG_g(A)$ is abelian, one direction is clear, so we are left to show that given an $A$-loop $a$ based at $x$ such that $a\genus 0_x$ then $a$ is $A$-homotopic to a product of commutators of $A$-loops. 

Let $h:T\Sigma\to A$ be an $A$-homology between the $A$-loop $a:I\to A$ and the zero $A$-loop $0_x$. Denote by $e_0\in\partial \Sigma$ the base point that maps to $x$ and choose smooth cycles $\gamma_i,\eta_i:I\to \Sigma$ based at $e$, such that $\{\partial \Sigma,\gamma_1,\dots,\gamma_{n},\eta_1,\dots,\eta_{n}\}$ is a set of generators for $\pi_1(\Sigma,e)$ ($n$ is the genus of $\Sigma$). Then $a_i:=h\circ\d \gamma_i$ and $b_i=h\circ \d \tau_i$ are $A$-paths. We can choose a smooth parameterization $\phi:\Delta\to \Sigma$, where $\Delta$ is a polygon with $4n+1$ sides, and assume that $\phi$ restricts on the sides to the generators $\gamma_i,\eta_i$ and to the base path of $a$. The composition $h\circ\d\phi$ determines an $A$-homotopy between $a$ and the concatenation:
\[ a_1\cdot b_1\cdot a_1^{-1} \cdot b_1^{-1} \dots \cdot a_{n}\cdot b_{n}\cdot a_{n}^{-1} \cdot b_{n}^{-1} . \]
This means that:
\[ a \sim (a_1,b_1)\cdots (a_{n},b_{n}), \]
and the result follows.
\end{proof}

The next example shows how to recover the classic Hurewicz homomorphism.

\begin{example}
Consider the tangent bundle $A=TM$ of a  manifold. Its Weinstein groupoid is the fundamental groupoid $\cG(TM)=\Pi_1(M)$ and so it genus integration is the abelianization $\cG_g(TM)=\Pi_1(M)^\ab$, as discussed in Example \ref{ex:fund:grpd}. The isotropy groups of $\Pi_1(M)^\ab$ are isomorphic to the first integer homology group of $M$, and the restriction of the morphism $p:\cG(TM)\to \cG_g(TM)$ to an isotropy group is just the usual Hurewicz homomorphism $\pi_1(M)\to H_1(M,\Z)$. 
\end{example}

\begin{example}
Let $\gg$ be a finite dimensional Lie algebra. Its Weinstein groupoid is just the 1-connected Lie group $G$ integrating $\gg$: 
\[ \cG(\gg)=G. \]
Because $G$ is 1-connected, the commutator subgroup $(G,G)$ is closed in $G$ and coincides with the 1-connected Lie group integrating $[\gg,\gg]$. Hence, the genus integration recovers the abelianization of the 1-connected integration (both set-theoretical and Lie) of the Lie algebra $\gg$:
\[ \cG_g(\gg)=G/(G,G)=\cG(\gg^\ab). \]
\end{example}

In the previous two examples the genus integration was a Lie groupoid, but of course this need not be the case. Here is a simple example.

\begin{example}
Consider the action Lie algebroid $A=\mathfrak{so}(3)\ltimes\R^3$ of Example \ref{ex:3dim:algbrd}. Its Weinstein groupoid is the action groupoid 
\[ \cG(A)={\mathrm{SU}}(2)\ltimes \R^3\tto \R^3, \]
where ${\mathrm{SU}}(2)$ acts on $\R^3$ via the covering ${\mathrm{SU}}(2)\to{\mathrm{SO}}(3)$. The isotropy groups for $x\not=0$ are isomorphic to $\Ss^1$ and at $x=0$ it is the group ${\mathrm{SU}}(2)$. Since $\Ss^1$ is abelian and $({\mathrm{SU}}(2),{\mathrm{SU}}(2))={\mathrm{SU}}(2)$, it follows that the genus integration of $A$ is the groupoid obtained from $\cG(A)$ by removing all non-identity arrows over $x=0$:
\[ \cG_g(A)=({\mathrm{SU}}(2)\ltimes (\R^3-\{0\}))\cup \{1_0\}. \]
This groupoid is not smooth.
\end{example}

\section{Extended monodromy}

The notion of extended monodromy was introduced recently in \cite{CFMb}, formalizing ideas from \cite{CM,Marcut}, as a mean to find obstructions to the existence of proper integrations of Poisson manifolds. It has also been used in \cite{AA17}. We consider here the general case of a Lie algebroid and show what exactly these groups obstruct. 

The extended monodromy of a Lie algebroid $A$ at $x\in M$ only depends on the restriction $A_L=A|_L$ of the Lie algebroid $A$ to the leaf $L$ through $x$. So consider a splitting $\sigma':TL\to A_L$ of the short exact sequence:
\[ 
\xymatrix{0\ar[r] & \gg_L\ar[r] & A_L\ar[r]^\rho & TL \ar[r]\ar@/^/[l]^{\sigma'} & 0,}
\]
This induces a splitting $\sigma:TL\to A^\ab_L$ for the the abelianization:
\[ 
\xymatrix{0\ar[r] & \gg^\ab_L\ar[r] & A_L^\ab\ar[r]^\rho & TL \ar[r]\ar@/^/[l]^{\sigma} & 0,}
\]
which has a curvature 2-form $\Omega\in\Omega^2(L,\gg^\ab_L)$ given by:
\[ \Omega(X,Y):=[\sigma(X),\sigma(Y)]-\sigma([X,Y]. \]
The splitting also induces a connection $\nabla$ on the bundle $\gg^\ab_L\to L$:
\[ \nabla_X\a:=[\sigma(X),\a]. \]
Note that two different splittings induce the same connection and the same curvature 2-form. Moreover, by the Jacobi identity, this connection is flat:
\begin{align*}
R^\nabla(X,Y)(\a)&=\nabla_X\nabla_Y\a -\nabla_Y\nabla_X\a-\nabla_{[X,Y]}\a\\
	    		   &=[\sigma(X),[\sigma(Y),\a]]-[\sigma(Y),[\sigma(X),\a]]-[\sigma([X,Y]),\a]\\
			   &=[[\sigma(X),\sigma(Y)]-\sigma([X,Y],\a]=0
\end{align*}
Hence $\gg^\ab_L\to L$ is canonically a flat bundle. Let us denote by $q:\tilde L^h\to L$ the holonomy cover of $L$ relative to $\nabla$, so that the pullback bundle $q^*\gg^\ab_L\to\tilde L^h$ is trivial with a canonical trivialization and connection.

\begin{definition}
Let $A\to M$ be a Lie algebroid. The {\bf extended monodromy homomorphism} at $x\in M$ is the homomorphism of abelian groups:
 \[ \partial_x^\ext:H_2(\tilde{L}^h)\to \cG(\gg_x^\ab), \quad [\gamma]\mapsto \exp\left(\int_\gamma q^*\Omega \right).\]
Its image $\cN^\ext_x(A)\subset \cG(\gg_x^\ab)$ is called the {\bf extended monodromy group} at $x$.
\end{definition}

Notice that the morphism $\partial_x^\ext$, and hence also its image $\cN_x^\ext$, only depends on the abelianization $A^\ab_L$.

Recall (see \cite[Section 3]{CF1}) that for an algebroid $A$ the (ordinary) {\bf monodromy homomorphism} at $x$  is given by a homomorphism
\[ \partial_x:\pi_2(L,x)\to \cG(\gg_x). \]
Its image $\cN_x(A)\subset \cG(\gg_x)$ is called the (ordinary) {\bf monodromy group} at $x$. Since $q:\tilde L^h\to L$ is a cover, it induces an isomorphism $\pi_2(L)\simeq\pi_2(\tilde L^h)$. Composing with the Hurewicz map we obtain a morphism $\pi_2(L,x)\to H_2(\tilde L^h)$ and we have:

\begin{proposition}
The extended and ordinary monodromy homomorphisms of a Lie algebroid fit into a commutative diagram:
\[
\xymatrix{
\pi_2(L,x)\ar[d]\ar[r]^{\partial_x}\ar[dr]^{\partial^\ab_x} & \cG(\gg_x)\ar[d]\\
H_2(\tilde{L}^h)\ar[r]_{\partial_x^\ext}& \cG(\gg_x^\ab)
}
\]
where $\partial^\ab_x$ denotes the monodromy homomorphism of the abelianization $A^\ab_L$. 
\end{proposition}

\begin{proof}
It follows from  \cite{CF1,CF3} that monodromy is functorial. In particular, this means that the surjective morphism $p:A_L\to A_L^\ab$ induces a commutative diagram:
\[
\xymatrix{
\pi_2(L,x)\ar@{=}[d]\ar[r]^{\partial_x} & \cG(\gg_x)\ar[d]\\
\pi_2(L,x)\ar[r]_{\partial_x^\ab}& \cG(\gg_x^\ab)
}
\]
where the bottom row is the monodromy homomorphism of the abelianization $A^\ab_L$. 

Since $A^\ab_L$ is abelian and the exponential map $\exp:\gg^\ab_x\to \cG(\gg_x^\ab)$ is surjective, it follows from \cite[Lemma 3.6]{CF1} that the monodromy of $A^\ab_L$ is given by:
\[ \partial^\ab_x:\pi_2(L,x)\to \cG(\gg_x^\ab), \quad [\gamma]\mapsto \exp\left(\int_\gamma \Omega \right).\]
Hence, the definition of extended monodromy shows that we also have a commutative diagram:
\[
\xymatrix{
\pi_2(L,x)\ar[d]\ar[r]^{\partial_x^\ab} & \cG(\gg_x^\ab)\ar@{=}[d]\\
H_2(\tilde{L}^h)\ar[r]_{\partial_x^\ext}& \cG(\gg_x^\ab)
}
\]
Composing the two diagrams, the result follows.
\end{proof}

Recall that the integrability obstructions for a Lie algebroid $A$ are expressed in terms of the discreteness of the monodromy groups $\cN_x(A)$. In particular, the transitive Lie algebroid $A_L$ is integrable iff $\cN_x(A)$ is discrete for some $x\in L$. In spite of the previous proposition, we observe that:
\begin{itemize}
\item $\cN_x^\ext(A)$ discrete $\Rightarrow$ $\cN_x(A^\ab_L)$ discrete;
\item It is possible for one of the groups $\cN_x(A)$ or $\cN_x^\ext(A)$ to be discrete, and the other one not;
\item It is possible for one of the groups $\cN_x(A)$ or $\cN_x(A^\ab)$ to be discrete, and the other one not.
\end{itemize}
The first statement follows from the definition of the extended monodromy. The remaining statements are illustrated in the examples below. So, in general, the only conclusion one can draw about integrability is:

\begin{corollary}
\label{cor:ext:integ}
If the extended monodromy of $A$ at some $x\in L$ is discrete  then $A^\ab_L$ is an integrable Lie algebroid.
\end{corollary}

The following examples are based in the following well-known construction. Let $\gg_L=\gg\times L\to L$ be a trivial bundle with fibre a Lie algebra $\gg$ and let $\Omega\in \Omega^2(L;Z(\gg))$ be a closed 2-form with values in the center of $\gg$. Then $A_L=TL\oplus \gg_L$ has a Lie algebroid structure with isotropy $\gg_L$, where the anchor is projection in the first factor and the bracket is given by:
\[ [(X,\xi),(Y,\eta)]=([X,Y],[\xi,\eta]_{\gg_M}+\Lie_X\eta-\Lie_Y\xi+\Omega(X,Y)). \]
The monodromy of the resulting algebroid $A_L$ at $x$ is:
\[ \cN_x(A_L)=\left\{\exp\left(\int_\gamma \Omega \right): [\gamma]\in\pi_2(L,x)\right\}\subset Z(\cG(\gg)). \]
On the other hand, the abelianization of $A_L$ is the algebroid $A^\ab_L=TL\oplus \gg^\ab_L$ with anchor projection on the first factor and Lie bracket given by a similar formula where $\Omega$ is replaced by $\Omega^\ab:=p\circ\Omega$, where $p:\gg\to\gg^\ab$ is the projection. In particular, we find that the extended monodromy of $A_L$ is given by:
\[ \cN_x^\ext(A_L)=\left\{\exp\left(\int_\gamma \Omega^\ab \right): [\gamma]\in H_2(L)\right\}\subset \cG(\gg^\ab). \]
The (ordinary) monodromy of $A^\ab_L$ is the subgroup $\cN_x(A^\ab_L)\subset \cN_x^\ext(A_L)$ obtained by restricting to classes $[\gamma]\in\pi_2(L)$.

\begin{example}
Let $\gg$ be the 3-dimensional Heisenberg Lie algebra defined by:
\[ [e_1,e_2]=[e_1,e_3]=0, \quad [e_2,e_3]=e_1. \]
We have:
\[ [\gg,\gg]=\R e_1=Z(\gg), \qquad \gg^\ab=\R e_2\oplus\R e_3. \] 
Now let $L=\Ss^2\times\Ss^2$ and consider the $Z(\gg)$-valued 2-form on $L$ given by:
\[ \Omega=\pr^*_1\omega_{\Ss^2}+\sqrt{2}\pr^*_2\omega_{\Ss^2}, \]
where $\omega_{\Ss^2}$ is an area form on $\Ss^2$ with total area 1, and $\pr_i:\Ss^2\times\Ss^2\to \Ss^2$ are the projections. We obtain a Lie algebroid $A_L$ for which:
\begin{align*}
\cN_x(A_L)&=\{\exp((n_1+n_2\sqrt{2})e_1): n_i\in \Z\}, \\ 
\cN_x(A^\ab_L)&=\cN_x^\ext(A_L)=\{1\}.
\end{align*}
So $A_L$ is not integrable, but $A^\ab_L$ is and the extended monodromy is discrete.
\end{example}

\begin{example}
Let $\gg$ be the 4-dimensional Lie algebra obtained by extending the 3-dimensional Lie algebra of the previous example by a central element $e_4$, so we have the extra relation:
\[ [e_i,e_4]=0, \quad (i=1,2,3). \]
We now have: 
\[ [\gg,\gg]=\R e_1, \quad Z(\gg)=\R e_1\oplus\R e_4,\quad \gg^\ab=\R e_2\oplus\R e_3\oplus\R e_4. \] 
Again let $L=\Ss^2\times\Ss^2$ and consider the $Z(\gg)$-valued 2-form on $L$ given by:
\[ \Omega=\pr^*_1\omega_{\Ss^2}(e_1+e_4)+\sqrt{2}\pr^*_2\omega_{\Ss^2}(e_1-e_4). \]
We now obtain a Lie algebroid $A_L$ for which:
\begin{align*}
\cN_x(A_L)&=\{\exp(n_1 (e_1+e_4)+n_2\sqrt{2}(e_1-e_4)): n_i\in \Z\}, \\ 
\cN_x(A^\ab_L)&=\cN_x^\ext(A_L)=\{\exp((n_1+n_2\sqrt{2})e_4): n_i\in \Z\}.
\end{align*}
So $A_L$ is integrable, but $A^\ab_L$ is not (and the extended monodromy is not discrete).
\end{example}

\begin{example}
We modify the previous example by taking $L=\T^2\times \T^2$ and replacing the form $\omega_{\Ss^2}$ by an area form $\omega_{\T^2}$ on $\T^2$ with total area 1. Then we see immediately that
\begin{align*}
\cN_x(A_L)&=\cN_x(A^\ab_L)=\{1\},\\ 
\cN_x^\ext(A_L)&=\{\exp((n_1+n_2\sqrt{2})e_4): n_i\in \Z\}.
\end{align*}
So $A_L$ and $A^\ab_L$ are integrable, but the extended monodromy is not discrete.
\end{example}

\begin{remark}
As we have observed before, the extended monodromy  groups were studied in \cite{CFMb} in the special case of cotangent bundles of regular Poisson manifolds. There it is shown that, in certain favorable cases, the extended monodromy can be computed by transverse variations of symplectic areas of compact surfaces of arbitrary genera, generalizing the result in \cite{CF2} for the (ordinary) monodromy of regular Poisson manifolds.
\end{remark}

\section{Smooth genus integrations}
\label{section:smooth:genus:integration}

In the previous section we saw that the discreteness of the extended monodromy implies the integrability of $A^\ab$, but not the converse. So it still remains to understand what the extended monodromy groups obstruct. In this section we will show that they control the smoothness of the genus integration of a cover of $A$. Henceforth we will assume that $A\to L$ is a \emph{transitive algebroid}. In particular, $A^\ab$ exists and it is obtained by factoring the commutators $[\gg_L,\gg_L]$.

We start with the following result that allows us to reduce the problem to the case of an abelian Lie algebroid.

\begin{proposition}
Let $A\to L$ be a transitive Lie algebroid. There is a commutative diagram of surjective groupoid morphisms of topological groupoids:
\[
\xymatrix{
\cG(A)\ar[r]\ar[d] & \cG(A^\ab)\ar[d]\\
\cG_g(A)\ar[r]^--{\simeq}& \cG_g(A^\ab)
}
\]
where the bottom arrow is an isomorphism.
\end{proposition}

\begin{proof}
The top arrow is the groupoid morphism $p_*:\cG(A)\to \cG(A^\ab)$ induced by the surjective algebroid morphism $p:A\to A^\ab$. By Proposition \ref{prop:morphism} in the Appendix, this is a surjective groupoid morphism whose kernel is the bundle of Lie groups $K\to M$ with fiber:
\[ K_x=(\cG(\gg_x),\cG(\gg_x))/\cN_x=(\cG(A)_x^0,\cG(A)_x^0), \]
where $\cG(A)_x^0=\cG(\gg_x)/\cN_x$ is the connected component of the identity of the isotropy group $\cG(A)_x$.

On the other hand, by Theorem \ref{thm:main:1}, the vertical arrows are the projections from the groupoids to the quotient by the commutators of the isotropy and hence are also surjective. Since $p_*$ maps commutators to commutators, it induces a surjective groupoid morphism $\cG_g(A)\to \cG_g(A^\ab)$ closing the square. Finally, this map is injective since $(\cG(A)_x^0,\cG(A)_x^0)\subset(\cG(A)_x,\cG(A)_x)$.
\end{proof}

The previous proposition together with the fact that the extended monodromy groups $\cN^\ext_x(A)$ of a transitive Lie algebroid $A\to L$ only depend on $A^\ab\to L$, allows us to make the assumption that $A$ is an abelian Lie algebroid. 

\subsection{The trivial holonomy case} 
The key to understand the relevance of the extended monodromy groups for the smoothness of $\cG_g(A)$ is the following. Recall \cite{CF1} that the ordinary monodromy group $\cN_x(A)$ is the kernel of the covering map $\cG(\gg_x)\to \cG(A)_x^0$, so that:
\[ \cG(A)^0_x= \cG(\gg_x)/\cN_x(A). \]
Composing the covering map $\cG(\gg_x)\to \cG(A)_x$ with the map $p:\cG(A)\to \cG_g(A)$, we obtain a surjective group morphism $\phi:\cG(\gg_x)\to \cG_g(A)_x^0$. When the holonomy of the canonical connection $\nabla$ on $\gg_L\to L$ is trivial, we have:

\begin{proposition}
\label{prop:genus:monodromy}
Let $A\to L$ be an abelian transitive Lie algebroid with trivial holonomy. The extended monodromy group $\cN_x^\ext(A)$ is the kernel of the morphism $\phi:\cG(\gg_x)\to \cG_g(A)^0_x$, so that:
\[ \cG_g(A)^0_x= \cG(\gg_x)/\cN_x^\ext(A). \]
\end{proposition}

\begin{proof}
Since $\gg_x$ is abelian, the exponential map $\exp:\gg_x\to\cG(\gg_x)$ is surjective. 
We need to show that $\exp(v)$ belong to the kernel of $\phi:\cG(\gg_x)\to \cG_g(A)^0_x$ if and only if:
\[ \exp(v)=\partial^\ext[\gamma], \]
for some compact surface $\gamma$.

Note that $\exp(v)$ belongs to the kernel of $q$ if and only if the $\gg_x$-path $a_0(t)=v$ is $A$-homologous to $0_x$, via some $A$-homology supported on  $\Sigma$ a square with some genus $n$. Since the base path of $a(t)=v$ is the constant path at $x$, the base map $\gamma$ of this $A$-homology has image a compact surface of genus $n$. We claim that:
\[ v=\int_\gamma \Omega, \]
so that $\exp(v)=\partial^\ext[\gamma]$. To see this we parametrize $\Sigma$ by a polygon with $4n+4$, by cutting along generators $\{\gamma_i,\eta_i\}$ of its fundamental group and its boundary, and we think of it as a homotopy $h:I\times I\to A$ between the boundary on $\Sigma$ and the closed loop $\tau=\prod_{i=1}^n(\gamma_i,\eta_i)$. Now choose a splitting $\sigma:TL\to A$ of the anchor $\rho:A\to TL$ such that:
\[ \sigma(\frac{\d \tau}{\d t})=h|_\tau. \]
and identify $\A$ with $TL\oplus \mathfrak{g}$ so the bracket becomes
\[ [(X, v), (Y, w)]= ([X, Y], [v, w]+ \nabla_{X}(w)- 
\nabla_{Y}(v)- \Omega(X, Y)) .\]
We choose a connection $\nabla^{L}$ on $L$, and we consider the 
connection $\nabla^A= (\nabla^{L}, \nabla)$ on $\A$. Note that
\[ T_{\nabla^A}((X, v), (Y, w))= (T_{\nabla^{L}}(X, Y), \Omega(X, Y)- [v, w]) \]
for all $X, Y\in TL$, $v, w\in \mathfrak{g}$. Then the $A$-homotopy above takes the form:
\[ h(t,\eps)=a(t,\eps)\d t+b(t,\eps)\d \eps, \]
where $a=(\frac{\d\gamma}{\d t}, \varphi)$, $b=(\frac{\d\gamma}{\d\epsilon}, \psi)$, and the $\gg_x$-paths $\varphi$ and $\psi$ satisfy the boundary conditions:
\[ \varphi(t,0)=v,\ \varphi(t,1)=0,\ \psi(0,\eps)=\psi(1,\eps)=0. \]
Since $[\varphi, \psi]=0$, equation \ref{diffeq} becomes the equation:
\[ \partial_{t}\psi- \partial_{\epsilon}\varphi= 
\Omega(\frac{d\gamma}{dt}, \frac{d\gamma}{d\epsilon}), \]
Integrating twice and using the boundary conditions, we find that:
\[ v=-\int_0^1\varphi(t,0)\, \d t=\int_0^1\int_0^1 \Omega(\frac{\d\gamma}{\d t}, \frac{\d\gamma}{\d\epsilon}) \, \d t\d \eps=\int_\gamma\Omega,\]
as claimed.

To prove the converse, we start with some compact surface $\gamma$ of genus $n$ in $L$ with a marked point $x$, representing some element in $H_2(L,\Z)$, and we show that $\int_\gamma \Omega$ is a (constant) $\gg_x$-path which is $A$-homologous to the trivial path:
\[ \int_\gamma \Omega \genus 0_x. \]
Again, we cut $\gamma$ as above so we can think of it as an ordinary homotopy between $x$ and $\tau=\prod_{i=1}^n(\gamma_i,\eta_i)$, We choose some splitting $\sigma$ and define $b(t,\eps)=\sigma(\frac{\d \gamma}{\d\eps})$. Then we solve equation \ref{diffeq} with initial conditions $a(0,t)=0$, obtaining a Lie algebroid map $h:=a(t,\eps)\d t+b(t,\eps)\d\eps:T(I\times I)\to A$. On the one hand, this can be thought of as an $A$-homology covering $\gamma$ between $0_x$ and the $\gg_x$-path $a(t,1)$, so that:
\[ a(\cdot,1)\genus 0_x. \]
On the other hand, an argument with a splitting as before, shows that:
\[ a(t,1)=\int_0^t \int_0^1 \Omega(\frac{\d\gamma}{\d t}, \frac{\d\gamma}{\d\epsilon}) \, \d\eps \d t. \]
But any $\gg_x$-path is $A$-homotopic to its average, so that:
\[ 0_x\genus a(\cdot,1)\sim \int_0^1 \int_0^1 \Omega(\frac{\d\gamma}{\d t}, \frac{\d\gamma}{\d\epsilon}) \, \d\eps \d t=\int_\gamma \Omega, \]
as claimed.
\end{proof}

We can now prove one of our main theorems:

\begin{theorem}
\label{thm:main:2}
Let $A\to L$ be a transitive Lie algebroid with trivial holonomy. The following statements are equivalent:
\begin{enumerate}[(a)]
\item the extended monodromy groups are discrete;
\item the genus integration $\cG_g(A)$ is smooth;
\item the abelianization $A^\ab$ has an abelian integration.
\end{enumerate}
Moreover, if any of these hold, then $\cG_g(A)$ has Lie algebroid isomorphic to $A^\ab$.
\end{theorem}

\begin{proof}
As we have observed before, we can assume that $A=A^\ab$ is abelian.
\medskip 

(a) $\Rightarrow$ (b) If we assume that the extended monodromy groups are discrete, then $A=A^\ab$ is an integrable algebroid, so $\cG(A)$ is smooth. By Proposition \ref{prop:genus:monodromy}, $\cG_g(A)$ is the quotient of $\cG(A)$ by the normal bundle of groups with fiber $\cN^\ext_x/\cN_x$. Since this bundle is discrete (=closed), we conclude that $\cG_g(A)$ is smooth and has Lie algebroid $A$.

\medskip 

(b) $\Rightarrow$ (a) Assume that the genus integration $\cG_g(A)$ is smooth. Then $\cG_g(A)_x^0$ is a Lie group and so the continuous group homomorphism $\phi:\cG(\gg_x)\to \cG_g(A)_x^0$ is a morphism of Lie groups. Hence, its kernel is a closed Lie subgroup of $\cG(\gg_x)$. By Proposition \ref{prop:genus:monodromy}, this kernel coincides with $\cN^\ext_x$, which therefore is a closed (=discrete) subgroup.
\medskip 

(a) $\Rightarrow$ (c) We already saw that (a) implies that $\cG_g(A)$ is smooth and has Lie algebroid $A=A^\ab$. So $\cG_g(A)$ is an abelian integration of $A^\ab$.
\medskip 

(c) $\Rightarrow$ (b)  Let $\cH$ be an abelian integration of $A=A^\ab$ with connected source fibers. Then $\cG(A)$ is smooth and, by the universal property, we have a commutative diagram of groupoid morphisms:
\[
\xymatrix{
\cG(A) \ar[r] \ar[d] & \cH \\
\cG_g(A) \ar[ru]
}
\]
In particular, looking at the isotropies groups, we obtain a commutative diagram:
\[
\xymatrix{
\cG(\gg_x)\ar[d]\\
\cG(A)_x^0 \ar[r] \ar[d] & \cH_x^0 \\
\cG_g(A)_x^0 \ar[ru]
}
\]
Since $\cG(\gg_x)$, $\cG(A)_x^0$ and $\cH_x^0$ are all Lie groups with the same Lie algebra $\gg_x$, the kernel of the map $\cG(\gg_x)\to \cH_x^0$ is a discrete subgroup. By the commutativity of the diagram and Proposition  \ref{prop:genus:monodromy}, this kernel contains $\cN_x^\ext$, which therefore is also discrete.
\end{proof}

\subsection{The non-trivial holonomy case} 
Let us now discuss how to get rid of the assumption of trivial holonomy. 

Given a transitive Lie algebroid $A\to L$ with non-trivial holonomy we can pull it to the holonomy cover:
\[
\xymatrix{
q^*A\ar[r] \ar[d] & A\ar[d] \\
\tilde L^h\ar[r]_q & L}
\]
The resulting Lie algebroid $q^*A\to \tilde L^h$ has trivial holonomy, so we deduce immediately the following version of Theorem \ref{thm:main:2}:

\begin{theorem}
\label{thm:main:2:hol}
Let $A\to L$ be a transitive Lie algebroid. The following statements are equivalent:
\begin{enumerate}[(a)]
\item the extended monodromy groups are discrete ;
\item the genus integration $\cG_g(q^*A)$ is smooth;
\item the abelianization $(q^*A)^\ab$ has an abelian integration.
\end{enumerate}
Moreover, if any of these hold then $\cG_g(q^*A)$ has Lie algebroid isomorphic to $(q^*A)^\ab$.
\end{theorem}

As an immediate corollary we obtain:
 
\begin{corollary}
If $A\to L$ is a transitive Lie algebroid whose monodromy and extended monodromy are both discrete, then the genus integration $\cG_g(q^*A)$ is the abelianization of the Weinstein groupoid $\cG(q^*A)$ in the smooth category.
 \end{corollary}

We still need to understand the  genus integration $\cG_g(A)$ of a Lie algebroid $A$ with non-trivial holonomy. If we look at an example with non-trivial holonomy we see that things are a bit more complicated.

\begin{example}
Consider the abelian Lie algebroid $A=T\Ss^3/\mathrm{O}(2)\to \mathbb{RP}^2$ discussed in Example \ref{ex:gauge:groupoid}. This Lie algebroid has no abelian integration, although the extended monodromy is discrete. In this example, $A$ has non-trivial holonomy since the isotropy bundle is the non-trivial line bundle over $\mathbb{RP}^2$:
\[ \gg_L=(\Ss^2\times\R)/\Z_2\to \mathbb{RP}^2. \]
Note that $H_2(\mathbb{RP}^2)=\Z$, so this Lie algebroid has discrete extended monodromy. As we saw in Example \ref{ex:gauge:groupoid}, the commutators $(\cG(A)_x,\cG(A)_x)\subset \cG(A)_x$ are closed subgroups and so the the genus integration $\cG_g(A)$ is the Lie groupoid 
\[ \cG_g(A)=\cG(A)^\ab=\Pi_1(\mathbb{RP}^2)\tto \mathbb{RP}^2,\] 
whose Lie algebroid is $T(\mathbb{RP}^2)$, and hence is not isomorphic to $A^\ab=A$.  

On the other hand, the holonomy cover of $\mathbb{RP}^2$ is $\Ss^2$ and we find that $q^*A$ is the Atiyah Lie algebroid of the Hopf fibration:
\[ q^*A=T\Ss^3/\mathrm{SO}(2)\to \Ss^2. \]
This Lie algebroid has Weinstein groupoid the gauge groupoid of the Hopf fibration:
 \[ \cG(q^*A)=(\Ss^3\times\Ss^3)/\mathrm{SO}(2)\tto \Ss^2, \]
an abelian groupoid, so we have $\cG_g(q^*A)=\cG(q^*A)$.
\end{example}

The assumption of trivial holonomy played a fundamental role in Proposition \ref{prop:genus:monodromy}: as the previous example shows, when the holonomy is not trivial, the kernel of of the morphism $\phi:\cG(\gg_x)\to \cG_g(A)^0_x$ in general does not coincide with the extended monodromy group $\cN_x^\ext(A)$. Notice that the closeness of this kernel still obstructs the existence of a smooth integration:

\begin{proposition}
\label{prop:genus:kernel}
Let $A\to L$ be an abelian transitive Lie algebroid. If $\cG_g(A)$ is smooth then the kernel of the morphism $\phi:\cG(\gg_x)\to \cG_g(A)^0_x$ is a closed Lie subgroup. Conversely, if $A$ is integrable and this kernel is closed then $\cG_g(A)$ is smooth.
\end{proposition}

\begin{proof}
Assume that the genus integration $\cG_g(A)$ is smooth. Then $\cG_g(A)_x^0$ is a Lie group and so the continuous group homomorphism $\phi:\cG(\gg_x)\to \cG_g(A)_x^0$ is a morphism of Lie groups. Hence, its kernel is a closed Lie subgroup of $\cG(\gg_x)$. 

Conversely, assume that the kernel of $\phi:\cG(\gg_x)\to \cG_g(A)_x^0$ is closed and that $A$ is integrable. Then $\cG(A)$ is smooth and $\cG_g(A)$ is the quotient of $\cG(A)$ by the normal bundle of groups with fiber $\Ker \phi/\cN_x$. Since this bundle is closed, we conclude that $\cG_g(A)$.
\end{proof}

Notice that the Lie algebroid of $\cG_g(A)$ will now be the quotient $A/\hh_L$, where $\hh_L\subset \gg_L$ is the bundle of Lie subalgebras with fiber $\hh_x$ the Lie algebra of $\Ker \phi_x$. So the remaining question is how to describe the kernel of $\phi:\cG(\gg_x)\to \cG_g(A)_x^0$ in terms of infinitesimal data when the holonomy of $A$ is not trivial. Let us give such a description and then we will justify it assuming that $A$ is integrable.

Let $q:\tilde L^h\to L$ be the holonomy cover of $A$. Let $\gamma:\Delta_{4n}\to L$ be a parameterized compact surface of genus $n$ in $L$, so $\gamma$ identifies the edges of the polygon $\Delta_{4n}$ in the usual manner. By a {\bf lift to $\tilde L^h$ of $\gamma$} we will mean a smooth map $\tilde{\gamma}:\Delta_{4n}\to \tilde L^h$ such that $\gamma=q\circ\tilde{\gamma}$. We have the following description of the kernel of $\phi:\cG(\gg_x)\to \cG_g(A)^0_x$ generalizing Proposition \ref{prop:genus:monodromy}:

\begin{proposition}
\label{prop:non-trivial:holonomy}
Let $A\to L$ be an integrable abelian transitive Lie algebroid. The kernel of the morphism $\phi:\cG(\gg_x)\to \cG_g(A)^0_x$ is given by:
\[ \Ker \phi_x=\left\{\exp\left(\int_{\tilde\gamma}\Omega\right): \textrm{ $\tilde{\gamma}$ a lift to $\tilde L^h$ of a compact surface $\gamma$ in $L$}\right\}. \]
\end{proposition}

\begin{proof}
Since we assume that $A$ is integrable, the groupoid $\cG(A)$ is isomorphic to the gauge groupoid of a principal $G$-bundle $P\to L$, with $P=\s^{-1}(x)$ 1-connected and $G=\cG(A)_x$. Note that $G$ is not connected in general, so it is not abelian, although both $\gg= \gg_x$ and $G^0$ are abelian, and this is crucial for what follows.

The isotropy bundle of $\cG(A)$ coincides with the adjoint bundle:
\[ \gg_L:=P\times_G\gg. \]
A choice of splitting $\sigma$ of the anchor of $A$ is the same thing as choice of a principal bundle connection $\theta\in\Omega^1(P;\gg)$, and it induces a flat connection $\nabla$ on $\gg_L$. Two principal bundle connections $\theta_1$ and $\theta_2$ induce the same flat connection on $\gg_L$, which corresponds to the fact that $\gg_L\to L$ is canonically a flat vector bundle. As usual, we let $q:\tilde L^h\to L$ denote the holonomy cover and by $\Omega\in\Omega^2(L,\gg_L)$ the curvature 2-form of the connection $\theta$, which is the same as the curvature of the splitting. Moreover, according to \cite[Thm II.8.2]{KN}, we can assume that the connection $\theta$ is ``fat'', i.e., that its holonomy coincides with the structure group $G$. 

The adjoint representation gives a short exact sequence:
\[ 
\xymatrix{ 1\ar[r] & K\ar[r] & G\ar[r]& \Ad(G) \ar[r] & 1}
\]
where the group $\Ad(G)$ is precisely the holonomy group of the flat connection $\nabla$. Note that $G^0\subset K\subset G$. This means that we can view $P$ both as principal $K$-bundle over $\tilde L^h$ or as principal $G$-bundle over $L$:
\[
\xymatrix{
 & P\ar[dl]\ar[dr] \\
 \tilde L^h \ar[rr]_q & & L}
\]
Moreover, the connection $\theta\in\Omega^1(P;\gg)$ is also principal $K$-bundle connection with curvature 2-form $q^*\Omega$. A simple application of Stokes formula leads to the following.

\begin{lemma}
\label{lem:holonomy}
Let $D\subset \tilde L^h$ be an embedded oriented 2-disc. Then:
\[ \exp\left(\int_D q^*\Omega \right)=\mathcal{T}_{\partial D},\]
where $\mathcal{T}_{\partial D}$ is $\theta$-parallel transport along the oriented curve $\partial D$.
\end{lemma}

Now we are ready to prove the proposition. The map $\phi:\cG(\gg_x)\to \cG_g(A)_x^0$ fits into the diagram:
\[
\xymatrix{
& G(\gg_x)\ar[d]\\
\gg_x\ar[r]_{\exp_G}\ar[ru]^{\exp}& G^0 \ar[d] \\
& \cG_g(A)_x^0
}
\]
where the exponentials are surjective (actually covering maps), since $\gg_x$ is abelian. By  Theorem \ref{thm:main:1}, the kernel of the map $G^0\to \cG_g(A)_x^0$ is
\[ (G,G)\cap G^0, \]
so the proposition will follow if we show that:
\begin{equation}
\label{eq:non-trivial:holonomy}
(G,G)\cap G^0=\left\{\exp_G\left(\int_{\tilde\gamma}\Omega\right): \textrm{ $\tilde{\gamma}$ a lift to $\tilde L^h$ of a compact surface $\gamma$ in $L$}\right\}.
\end{equation}

Let $(g_1,h_1)\cdots (g_n,h_n)\in G^0$ with $g_i,h_i\in G$. Since the connection is fat, we can choose horizontal paths in $P=\s^{-1}(x)$ connecting successively the points:
\[ 1_x,~g_1,~g_1h_1,~g_1h_1 g_1^{-1},~(g_1,h_1),~(g_1,h_1)g_2,~(g_1,h_1)g_2h_2,~\dots~,(g_1,h_1)\cdots (g_n,h_n). \] 
Finally, we connect the last point back with $1_x$ using a curve in $G^0$. 
 \begin{figure}[h]
    \centering
    \includegraphics[height=5cm]{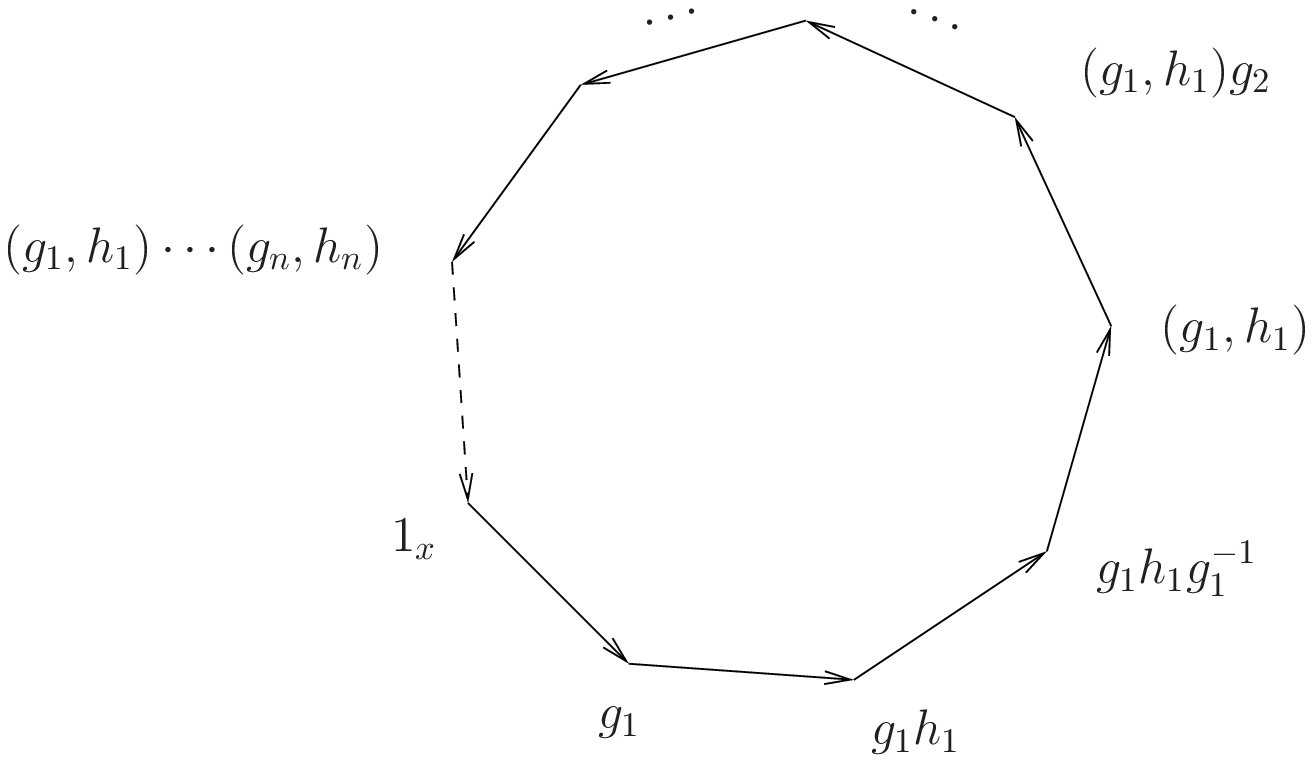}
    \label{GHomotopy}
\end{figure}

Since $P$ is 1-connected, there is a  smooth map $\bar{\gamma}:\Delta_{4n} \to P$ from the $4n$-gon, such that the first $4n-1$ sides of $\Delta_{4n}$ are mapped to first $4n-1$ horizontal curves, while the last side of $\Delta_{4n}$ is mapped to the concatenation of the last two curves. Composing $\bar{\gamma}:\Delta_{4n} \to P$ with the projection $P\to L$, we obtain a smooth map $\gamma:\Delta_{4n}\to L$ where the sides of the $4n$-gon are identified as in a surface of genus $n$. In other words, $\gamma$ parameterizes a closed surface $\Sigma$ of genus $n$. On the other hand, composing $\bar{\gamma}:\Delta_{4n} \to P$ with the projection $P\to \tilde L^h$ gives us a lift $\tilde{\gamma}::\Delta_{4n} \to \tilde L^h$ of $\gamma$. Moreover, by the lemma above we have:
if $a_i$ and $b_i$ are the sides of $\gamma$, we obtain:
\[ \exp_G\left(\int_{\tilde\gamma}\Omega\right)=\prod_{i=1}^{n} \left(\mathcal{T}_{a_i},\mathcal{T}_{b_i}\right)=(g_1,h_1)\cdots (g_n,h_n). \]
This shows the inclusion $\subset$ in \eqref{eq:non-trivial:holonomy}. 

We claim that the inclusion $\supset$ in \eqref{eq:non-trivial:holonomy} also holds: if we are given $\tilde{\gamma}$, a lift to $\tilde L^h$ of a compact surface $\gamma$ in $L$, then the sides of $\tilde{\gamma}$ project to curves $a_i$ and $b_i$, which are cycles generating the fundamental group of $\gamma$. The lemma then shows again that:
\[ \exp_G\left(\int_{\tilde\gamma}\Omega\right)=\prod_{i=1}^{n} \left(\mathcal{T}_{a_i},\mathcal{T}_{b_i}\right)\in (G,G)\cap G^0, \]
as claimed.
\end{proof}

One should be able to adapt the technique of the proof of Proposition \ref{prop:genus:monodromy} to show that Proposition \ref{prop:non-trivial:holonomy} also holds in the non-integrable case.

\begin{remark}
It is worth to see what Lemma \ref{lem:holonomy} says about the extended monodromy when $A^\ab$ is integrable. In this case, taking again for $P$ the source fiber $\s^{-1}(x)$ of $\cG(A^\ab)$, we conclude that if $[\gamma]\in H_2(\tilde L^h,\Z)$ is represented by a closed surface $\gamma\subset \tilde L^h$ of genus $n$, and we let $a_1,\dots,a_{n}$, $b_1,\dots,b_{n}$ be cycles based at some $x\in \gamma$ generating $\pi_1(\gamma,x)$, then:
\[  \exp_G\left(\int_{\gamma} q^*\Omega_\theta\right)=\prod_{i=1}^{n} \left(\mathcal{T}_{a_i},\mathcal{T}_{b_i}\right)=\mathcal{T}_{\prod_{i=1}^{n} (a_i,b_i)}. \]
It follows that we have a commutative diagram:
\[
\xymatrix@C=4pt{
\pi_2(L,x)\ar[drrrr]^\partial\ar[d]_h\\ 
H_2(\tilde L^h,\Z)\ar[d]\ar[rrrr]_{\partial^\ext}& &&&\cG(\gg^\ab_x) \ar[d]\\
\coKer h\ar[rrrr]_{\partial^P} && &&G^0 \ar@{=}[r] &\cG(A^\ab)_x^0
}
\]
The previous formula says that the induced map $\partial^P$ only depends on homotopy classes of paths in $\tilde L^h$. This can be thought of as an illustration of Hopf's Theorem stating that the cokernel of the Hurewicz map $h:\pi_2(\tilde L^h)\to H_2(\tilde L^h,\Z)$ is the 2nd group homology $H_2(\pi_1(\tilde L^h),\Z)$.
\end{remark}

\subsection{The proper case}
The extended monodromy groups were introduced in \cite{CFMb} in the special case of cotangent Lie algebroids of regular Poisson manifolds, as obstructions to the existence of proper integrations. Our results have the following corollary, which generalizes a result of \cite{CFMb} for the special case of cotangent Lie algebroids of regular Poisson manifolds:

\begin{theorem}
\label{thm:proper}
If a Lie algebroid $A\to M$ admits an s-proper integration $\cG\tto M$ then the extended monodromy groups are discrete.
 \end{theorem}

For the proof, we will use the following result.

\begin{lemma}
\label{lem:closed:commutator}
If $G$ is a compact Lie group then the commutator subgroup $(G,G)\subset G$ is a closed subgroup.
\end{lemma}

\begin{proof}
By \cite[Chp.~3, $\S 9$, Exer.~9]{Bourbaki}, for any Lie group there is a fixed integer $p'$ and a neighborhood of the identity $U$ such that any element of $(G,G)\cap U$ is a product of $p'$ commutators. For a compact Lie group $G$, the connected component $G^0$ is the union of a finite number of powers $U^n$ and the quotient $G/G^0$ is finite, so it follows that there is a fixed integer $p$ such that any element in $(G,G)$ is a product of $p$ commutators. 

Now given a sequence $g_n=(a_n^1,b_n^1)\cdot (a_n^2,b_n^2)\dots \cdot (a_n^p,b_n^p)\in (G,G)$ such that $g_n\to g\in G$, by compactness, we can find convergent subsequences $a_{n_k}^i\to a^i$ and  $b_{n_k}^i\to b^i$, and it follows that $g=(a^1,b^1)\cdot (a^2,b^2)\dots \cdot (a^p,b^p)\in (G,G)$. This shows that $(G,G)$ is closed, so the lemma holds.
\end{proof}

\begin{proof}[Proof of Theorem \ref{thm:proper}]
Let $\cG\tto M$ be an s-proper integration $A\to M$. The restriction $\cG_L\tto L$ to a leaf $L$ of $A$ is also an s-proper Lie groupoid and the extended monodromy groups $\cN_x$, with $x\in L$, only depend on this restriction. So we can assume that $A\to L$ is transitive.

By the lemma, the group bundle $(G_L^0,G^0_L)$ formed by the commutators $(\cG_x^0,\cG_x^0)$, $x\in L$, is a closed, normal, subbundle of $\cG$ integrating the bundle of Lie algebras $[\gg_L,\gg_L]\to L$. Therefore, the quotient $\cG/(G_L^0,G^0_L)\to L$ is an s-proper Lie groupoid integrating the abelianization $A^\ab\to L$. This means that it is enough to prove the theorem if $A$ is abelian and transitive.

If the holonomy of $A\to L$ was trivial we would be done: we take an s-proper integration $\cG\tto L$ of $A$ and by the lemma  $(G_L,G_L)=\cup_{x\in L}(\cG_x,\cG_x)\subset\cG$ is a discrete, normal, subgroupoid, so the quotient 
\[ \cG/(G_L,G_L)\tto L, \]
is an abelian integration of $A_L^\ab$. By Theorem \ref{thm:main:2} the extended monodromy is discrete.

If the holonomy of $A\to L$ is not trivial, we observe that the existence of an s-proper integration implies that the holonomy is finite: since $G^0$ is abelian and $G$ is compact, it follows that $\Ad(G)$ is finite, and so is the holonomy of $\nabla$ (see the proof of Proposition \ref{prop:non-trivial:holonomy}). Hence, the holonomy cover $q:\tilde L^h\to L$ is finite and the pullback groupoid $q^*\cG$ is an s-proper integration of $q^*A\to L$. Since $q^*A$ has trivial holonomy, we can apply the previous argument and Theorem \ref{thm:main:2:hol} to conclude that the extended monodromy is discrete.
\end{proof}

{
\subsection{An example: prequantization and genus integration}

Let $\omega\in\Omega^2(M)$ be a symplectic form\footnote{The non-degeneracy condition actually plays no role in the following discussion, so one can assume that $\omega$ is just a closed 2-form.} and consider its group of periods
\[ \Per(\omega):=\{\int_\gamma \omega:[\gamma]\in H_2(M,\Z)\}\subset \R. \]
Then the {\bf prequantization condition} for $\omega$ is:
\[ \Per(\omega)=k\Z, \]
for some integer $k$. This condition is equivalent to the existence of a principal $\Ss^1$-bundle $\pi:P\to M$ with a connection $\theta\in \Omega^1(M,\R)$ such that:
\[ \pi^*\omega=\d\theta. \]
This fact is well-known, but we will give a new geometric proof below using the genus integration (see Theorem \ref{thm:prequantization}).

In this discussion there is implicit a normalization, namely that $\Ss^1=\R/\Z$. If we allow for any normalization $\Ss^1_a=\R/a\Z$, then we can rephrase the prequantization condition for $\omega$ as:
\[  \Per(\omega)\subset \R\ \text{\bf is a discrete subgroup}. \]
Indeed, this holds if and only if $\Per(\omega)=a\Z\subset\R$ for some $a\in\R$, and then this condition is equivalent to the existence of a principal $\Ss^1_a$-bundle $\pi:P\to M$ with a connection $\theta\in \Omega^1(M,\R)$ such that:
\[ \pi^*\omega=\d\theta. \]
In this more general form the prequatization condition allows for $a=0$, in which case $\Ss^1_0=\R$ (notice that we do not assume $M$ compact).

The prequantization condition in this more general form can be viewed as a condition on the extended monodromy of the so-called {\bf prequantization Lie algebroid} $A_\omega\to M$ (see, e.g., \cite{Cr}). This is a special case of the construction described after Corollary \ref{cor:ext:integ}, namely it is the central extension determined by $\omega$, viewed as a Lie algebroid 2-cocycle on $A=TM$: it is supported in the vector bundle
\[ A_\omega:=TM\oplus \R,\] 
with anchor $\rho:A\to TM$ the projection on the first factor and bracket given by:
\[ [(X,f),(Y,g)]:=([X,Y],X(g)-Y(f)+\omega(X,Y)). \]
Notice that $A_\omega$ is an abelian, transitive, Lie algebroid, so $A^\ab_\omega=A_\omega$. Its isotropy bundle is the trivial line bundle:
\[ \gg_M=M\times\R, \]
and the connection induced on $\gg_M$ by the canonical splitting $\sigma:TM\to A_\omega$ is the trivial connection:
\[ \nabla^\sigma_X f=\Lie_X f. \]
The curvature of this splitting is just the original 2-form:
\[ \Omega^\sigma=\omega, \]
so we find that the monodromy groups of $A_\omega$ coincide with the group of spherical periods of $\omega$:
\[ \cN_x=\SPer(\omega):=\{\int_\gamma \omega:[\gamma]\in\pi_2(M)\}\subset \R. \]
On the other hand, the holonomy of $\nabla^\sigma$ is trivial, so the extended monodromy groups of $A_\omega$ coincide with the group of periods of $\omega$:
\[ \cN_x^\ext=\Per(\omega)=\{\int_\gamma \omega:[\gamma]\in H_2(M,\Z)\}\subset \R. \]

Our results immediately lead to the following result:

\begin{theorem}
\label{thm:prequantization}
Let $\omega\in\Omega^2(M)$ be a closed 2-form. The following conditions are equivalent:
\begin{enumerate}[(i)]
\item $\omega$ is prequantizable: $\Per(\omega)=a\Z\subset\R$;
\item the genus integration of $A_\omega$ is smooth;
\item there is a principal $\Ss^1_a$-bundle $\pi:P\to M$ with connection $\theta\in \Omega^1(M,\R)$ such that:
\[ \pi^*\omega=\d\theta. \]
\end{enumerate}
\end{theorem}

\begin{proof}
As we observed above, the holonomy of $A_\omega$ is trivial, and the extended monodromy groups of $A_\omega$ coincide with the group of periods of $\omega$. Hence, the equivalence of (i) and (ii) follows from from Theorem \ref{thm:main:2:hol}.

If (iii) holds, then the gauge groupoid $\cG:=(P\times P)/\Ss^1_a\tto M$ is an abelian integration of $A_\omega$. Hence, by Theorem \ref{thm:main:2:hol}., (ii) holds. Conversely, assume that (ii) holds, so $\cG_g(A_\omega)\tto M$ is a transitive groupoid with abelian isotropy groups $\cG_g(A_\omega)_x$. Its connected component of the identity is the group $\Ss^1_a$, where $\Per(\omega)=a\Z\subset\R$. Since $\Ss^1_a$ is a divisible group, we have $\text{Ext}^1(A,\Ss^1_a)=0$, for any abelian group $A$. In particular, we have 
\[ \cG_g(A_\omega)_x\simeq \Ss^1_a\times D, \] 
where $D$ is a discrete group. So $\cG_g(A_\omega)_x$ contains a discrete, normal subgroup $D$, and we can quotient by $D$ to obtain a groupoid $\cG\tto M$, still integrating $A_\omega$, but with isotropy groups $\Ss^1_a$. A source fiber $s^{-1}(x)$ of this groupoid gives the desired principal circle bundle $P\to M$, and the canonical splitting $\sigma$ induces the desired connection $\theta$ on $P$.
\end{proof}

\begin{remark}
Note that it is possible to choose a manifold $M$ and a closed 2-form $\omega$ such that $\SPer(\omega)$ is discrete and $\Per(\omega)$ is non-discrete: for example, one can take $M=\Ss^1\times\Ss^1\times\Ss^1\times\Ss^1$ and 
\[ \omega=\d\theta_1\wedge\d\theta_2+\sqrt{2}\d\theta_3\wedge\d\theta_4. \]
In such case, we have that $A_\omega$ is integrable but it admits no abelian integration. Hence, the previous proposition is the correct extension of the results of \cite{Cr} to the non-simply connected case. In the simply connected case, the groups of periods and spherical periods coincide and we recover the results of \cite{Cr}.
\end{remark}
}
\section{Some open questions}

Our results on the genus integration raise many interesting questions which we reserve for future work. Here is a list of the ones, we believe, are more significant:
\begin{enumerate}[Q1.]
\item What are the obstructions to the existence of an abelianization of a non-transitive Lie algebroid/groupoid?
\item How does abelianization behave under Morita equivalence?
\item When $A=T^*M$ is the cotangent Lie algebroid of a Poisson manifold $(M,\pi)$, and $\cG_g(T^*M)$ is smooth, is it a symplectic groupoid?
\item Still for cotangent Lie algebroids of Poisson manifolds, is there a genus version of the Poisson sigma-model that leads to $\cG_g(T^*M)$?
\end{enumerate}

As pointed out by one referee, an alternative natural way of fitting in higher genus surfaces into the integration problem is to keep the Weinstein groupoid $\cG(A)$ and use open surfaces with boundary as 2-morphisms, thus promoting $\cG(A)$ to a 2-groupoid. In this way, the genus of the surface is remembered. Actually, the genus integration should be the first layer of higher structures involving genus homotopies (or cobordism) based on higher dimensional manifolds, instead of just surfaces. We hope to be able to develop these 
higher category aspects in future work.

\appendix
\section{On integration of algebroid morphisms}

In Section \ref{section:smooth:genus:integration} we have used the following general result about integration of algebroid morphisms. Since we are not aware of a reference for this result and it maybe of independent interest, we have included a detailed proof. Note that we do not make any integrability or transitivity assumptions.

\begin{proposition}
\label{prop:morphism}
Let $A\to M$ and $B\to M$ be Lie algebroids and let $p:A\to B$ be a surjective algebroid morphism covering the identity. Then the induced groupoid morphism 
\[ p_*:\cG(A)\to \cG(B), \quad [a]\mapsto [p(a)], \]
is surjective and has kernel the bundle of groups $K\to M$ whose fiber at $x\in M$ is given by:
\[ K_x=\widetilde{K}_x/\cN_x(A), \]
where  $\cN_x(A)\subset \cG(\gg_x(A))$ is the monodromy group of $A$ at $x$ and $\widetilde{K}_x\subset \cG(\gg_x(A))$ is the connected Lie subgroup integrating the Lie subalgebra $\Ker p_x\subset \gg_x(A)$.
\end{proposition}

\begin{proof}
Clearly, the morphism $p_*$ is surjective. To determine its kernel, we will make use of the following, which follows from the results in \cite{CF1}:
\begin{enumerate}[(a)]
\item The connected component of the identity of the isotropy group $\cG(A)_x$ consists of $A$-paths whose base path is a contractible loop based at $x$ and is given by:
\[ \cG(A)_x^0=\cG(\gg_x(A))/\cN_x(A), \]
\item For a surjective algebroid morphism $p:A\to B$ covering the identity the monodromy morphisms fit into a commutative diagram:
\[
\xymatrix{
 & \cG(\gg_x(A))\ar[dd]^{(p_x)_*}\\
\pi_2(L,x)\ar[ru]^{\partial_x^A}\ar[rd]_{\partial_x^B}\\
 & \cG(\gg_x(B))}
\]
\end{enumerate}

From (a) and (b) it follows that  the restriction of $p_*:\cG(A)\to \cG(B)$ to the connected component of the isotropy group
\[ (p_*)_x:  \cG(A)_x^0\to  \cG(B)_x^0, \]
is obtained by first integrating the Lie algebra morphism $p_x:\gg_x(A)\to \gg_x(B)$ to a morphism of the 1-connected Lie groups:
\[ (p_x)_*: \cG(\gg_x(A))\to \cG(\gg_x(B)), \]
and then quotienting this morphism by the monodromy groups:
\[ (p_*)_x=[(p_x)_*]: \cG(\gg_x(A))/\cN_x(A)\to \cG(\gg_x(B))/\cN_x(B). \]
In particular, it follows that the kernel of $p_*$ at $x$ is given by:
\[ K_x=\widetilde{K}_x/\cN_x(A), \]
where $\widetilde{K}_x\subset \cG(\gg_x(A))$ is the kernel of $(p_x)_*$. But by standard Lie theory, this is the connected Lie subgroup integrating the Lie subalgebra $\Ker p_x\subset \gg_x(A)$, so the proposition follows.
\end{proof}

\end{document}